\documentclass[oneside, reqno]{amsart}

\usepackage{amsmath}
\usepackage{amssymb}
\usepackage{bm}
\usepackage{graphicx}
\usepackage{tikz}
\usepackage{cleveref}
\usepackage{mathtools}

\usetikzlibrary{calc}
\usetikzlibrary {arrows.meta}

\newtheorem{theorem}{Theorem}[section]
\newtheorem{lemma}[theorem]{Lemma}
\newtheorem{proposition}[theorem]{Proposition}

\theoremstyle{definition}
\newtheorem{definition}[theorem]{Definition}
\newtheorem{example}[theorem]{Example}
\newtheorem{remark}[theorem]{Remark}
\newtheorem{conjecture}[theorem]{Conjecture}

\newtheorem{question}[theorem]{Question}

\numberwithin{equation}{section}

\newcommand{\A}{\mathcal{A}}

\newcommand{\F}{\mathbb{F}}
\newcommand{\Z}{\mathbb{Z}}

\newcommand{\R}{\mathbb{R}}
\newcommand{\C}{\mathbb{C}}

\DeclareMathOperator{\rank}{rank}
\DeclareMathOperator{\codim}{codim}
\DeclareMathOperator{\Sep}{Sep}
\DeclareMathOperator{\Supp}{Supp}

\newcommand{\Alt}{\operatorname{Alt}}
\newcommand{\Aut}{\operatorname{Aut}}
\newcommand{\covec}{\operatorname{\mathcal{V}^*}}
\newcommand{\cha}{\operatorname{char}}
\newcommand{\Fil}{\operatorname{\mathsf{F}}}

\newcommand{\Heav}{\mathcal{H}}
\newcommand{\gHeav}{\widetilde{\mathcal{H}}}
\newcommand{\Idem}{\operatorname{Idem}}
\newcommand{\OS}{\operatorname{\mathsf{OS}}}

\newcommand{\PrimIdem}{\operatorname{PrimIdem}}
\newcommand{\sgn}{\operatorname{sgn}}
\newcommand{\circuit}{\operatorname{\mathcal{C}}}
\newcommand{\Tope}{\operatorname{\mathcal{T}}}
\newcommand{\gTope}{\operatorname{\widetilde{\mathcal{T}}}}
\newcommand{\vecto}{\operatorname{\mathcal{V}}}
\newcommand{\VG}{\operatorname{\mathcal{VG}}}
\newcommand{\grVG}{\operatorname{\mathsf{VG}}}

\newcommand{\ch}{\operatorname{ch}}
\newcommand{\eps}{\varepsilon}
\newcommand{\sqzero}{\operatorname{\mathsf{N}^2}}

\title[Varchenko-Gelfand algebras]{Reconstruction of oriented matroids from Varchenko-Gelfand algebras}

\author{Yukino Yagi}
\address{Department of Mathematics, The University of Osaka, Osaka, Japan}
\email{u401713d@ecs.osaka-u.ac.jp}

\author{Masahiko Yoshinaga}
\address{Department of Mathematics, The University of Osaka, Osaka, Japan}
\email{yoshinaga@math.sci.osaka-u.ac.jp}

\date{23 September 2025}

\begin{document}

\begin{abstract}
The algebra of $R$-valued functions on the set of chambers of a real hyperplane arrangement is called the Varchenko–Gelfand (VG) algebra. 
This algebra carries a natural filtration by the degree with respect to Heaviside functions, giving rise to the associated graded VG algebra. 
When the coefficient ring $R$ is an integral domain of characteristic $2$, the graded VG algebra is known to be isomorphic to the Orlik–Solomon algebra. 
In this paper, we study VG algebras over coefficient rings of characteristic different from $2$, and investigate to what extent VG algebras determine the underlying oriented matroid structures. 

Our main results concern hyperplane arrangements that are generic in codimension $2$. 
For such arrangements, if $R$ is an integral domain of characteristic not equal to $2$, then the oriented matroid can be recovered from both the filtered and the graded VG algebras. 
As a byproduct, we prove that, unlike the complexification, 
the cohomology ring of the complement of a $3$-plexification of a 
real arrangement is not determined by the intersection lattice. 

We also formulate an algorithm that is expected to reconstruct oriented matroids from VG algebras in the case of general arrangements.
\end{abstract}

\maketitle 

\tableofcontents

\section{Introduction}
\label{sec:intro}

Varchenko and Gelfand \cite{var-gel} introduced the algebra 
$\VG(\A)_R$ of $R$-valued functions on 
the set of chambers of a real central hyperplane arrangement $\A$. 
The VG algebra is generated by Heaviside functions, which define 
a degree filtration $\Fil^\bullet\VG(\A)$. 
Passing to the associated graded algebra, one obtains 
$\grVG^\bullet(\A)=\operatorname{Gr}_{\Fil^\bullet}^\bullet\VG(\A)$. 

These algebras are closely related to cohomology rings of 
certain spaces. Bj\"orner \cite{bjo-sub} introduced 
$c$-plexification $\A\otimes\R^c$ and its complement $M_c(\A)$ 
for $c>0$. Note that $\A\otimes\R=\A$ and $\A\otimes\R^2$ is the 
complexification. Varchenko and Gelfand proved that the 
Hilbert series of $\grVG^\bullet(\A)$ coincides with 
the Poincar\'e polynomial of the complexified 
complement $M_2(\A)$. 
Moreover, if $R=\F_2$, the graded VG algebra 
$\grVG^\bullet(\A)_{\F_2}$ is isomorphic to 
the cohomology ring $H^\bullet(M_2(\A), \F_2)$, which is also 
isomorphic to the Orlik-Solomon algebra $\OS^\bullet(\A)_{\F_2}$ 
defined from the intersection lattice $L(\A)$. 
Furthermore, Moseley \cite{mos-equ} proved that the graded VG algebra 
$\grVG^\bullet(\A)$ is isomorphic to the cohomology ring 
$H^\bullet(M_3(\A), R)$ of the complement of $\A\otimes\R^3$, which was 
recently generalized to arrangements in convex domains 
\cite{dorp-vg, dbpw, mos-equ, pro}.

On the other hand, the arrangement $\A$ determines 
an oriented matroid. 
The oriented matroid  is a very strong invariant: 
for example, the homeomorphism type of the complexified complement 
$M_2(\A)$ is determined by the oriented matroid of $\A$ \cite{bz-st}. 
VG algebras are clearly oriented-matroid invariants. 
Indeed, if the oriented matroids of $\A_1$ and $\A_2$ are 
isomorphic (equivalent up to reorientation and permutation 
of the ground set, see \S \ref{sec:prel} for details), 
then the corresponding VG algebras are also isomorphic. 

It is therefore natural to ask whether VG algebras over 
coefficient rings with $\cha R\neq 2$ contain more information 
than in the $\F_2$-case. 
We expect that VG algebras encode information 
as rich as that of oriented matroids. 
The most ambitious conjecture in this direction is the following. 

\begin{conjecture}
\label{conj:intro}
(Conjecture \ref{conj:reconst}) 
Let $R$ be an integral domain with $\cha R\neq 2$, and let 
$\A_1$ and $\A_2$ be real arrangements. 
If $\grVG^\bullet(\A_1)_R\simeq\grVG^\bullet(\A_2)_R$ as 
graded $R$-algebras, then the oriented matroids of 
$\A_1$ and $\A_2$ are isomorphic. 
\end{conjecture}

The main results of this paper establish reconstructibility 
under the assumption that $\A$ is generic in codimension $2$ 
(See Definition \ref{def:gencodim2}, which is also equivalent  to the statement that 
the intersection of any three distinct hyperplanes has codimension $3$.)

\begin{theorem}[Theorem \ref{thm:gencodim2}]
\label{thm:12}
Let $R$ be an integral domain with $\cha R\neq 2$, and 
let $\A_1$ and $\A_2$ be real arrangements. 
Assume that $\A_1$ is generic in 
codimension $2$ (see Definition \ref{def:gencodim2}). 
If $\VG(\A_1)_R\simeq\VG(\A_2)_R$ as 
filtered $R$-algebras, then the oriented matroids of 
$\A_1$ and $\A_2$ are isomorphic. 
\end{theorem}

\begin{theorem}[Theorem \ref{thm:main2}]
\label{thm:13}
Let $R$ be an integral domain with $\cha R\neq 2$, and 
let $\A_1$ and $\A_2$ be real arrangements. 
Assume that $\A_1$ is generic in codimension $2$. 
If $\grVG^\bullet(\A_1)_R\simeq\grVG^\bullet(\A_2)_R$ as 
graded $R$-algebras, then the oriented matroids of 
$\A_1$ and $\A_2$ are isomorphic. 
\end{theorem}
Clearly Theorem \ref{thm:13} implies Theorem \ref{thm:12} immediately. 
However, the proofs rely on different techniques. 
The proof of 
Theorem \ref{thm:12} reconstructs the tope graph of the arrangement 
by means of primitive idempotent elements of $\VG(\A)$, 
while the proof of 
Theorem \ref{thm:13} reconstructs the signed circuits 
of the arrangement 
using square-zero elements of $\grVG^1 (\A)$ 
and the relations among them. 
For this reason, we present the two results separately. 

These results show that VG algebras contain 
sufficient information to recover oriented matroids 
under genericity conditions. We also point out that 
none of the following invariants determine the VG algebra: 
the intersection lattice $L(\A)$, 
the Orlik-Solomon algebra $\OS^\bullet(\A)$, or 
the Poincar\'e polynomial of $M_2(\A)$ 
(Example \ref{ex:6planes}, Example \ref{ex:disting}, Example \ref{ex:LtoGRVG}). 
These relationships are summarized in Figure \ref{fig:implications}. 
(See also \S \ref{subsec:mot} for details.) 
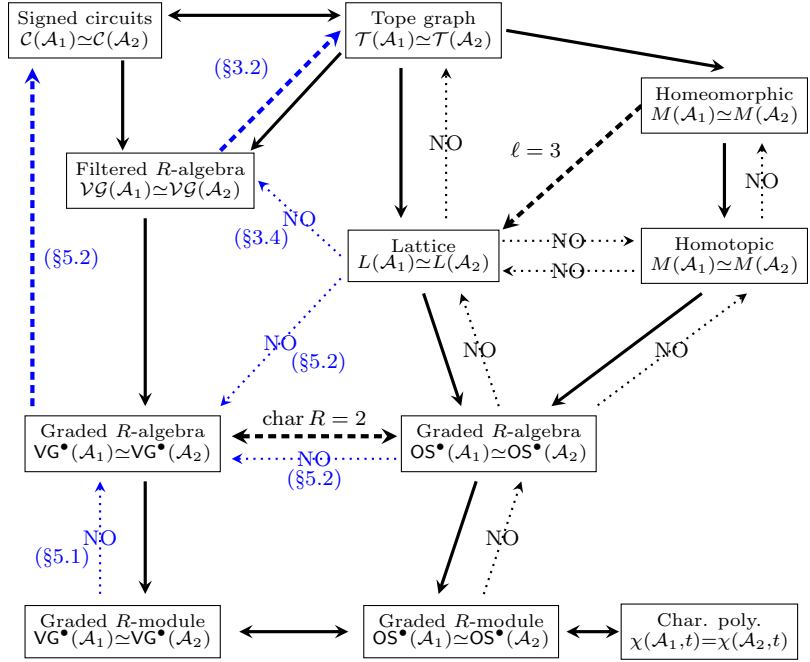
\begin{figure}[htbp]
\centering
\begin{tikzpicture}

\coordinate (CIR) at (-0.5,8); 
\coordinate (VG) at (0.5,6); 
\coordinate (GRVG) at (0,2.5); 
\coordinate (GRAB) at (0,0); 
\coordinate (OM) at (4,8); 
\coordinate (MAT) at (4,5); 
\coordinate (OS) at (4.5,2.5); 
\coordinate (GROS) at (4,0); 
\coordinate (BET) at (4,-1); 
\coordinate (HOMEO) at (7.5,7); 
\coordinate (HTP) at (7.5,5); 
\coordinate (CHAR) at (7.3,0); 
\coordinate (M3) at (-3.5,5); 
\coordinate (COM3) at (-3.5,2.5); 
\coordinate (COM2) at (7.7,2.5);

\draw[double, stealth-stealth, thick] ($(CIR)+(1.1,0.2)$) -- ($(OM)+(-1.1,0.2)$); 
\draw[double, -stealth, thick] ($(CIR)+(0.5,-0.4)$) -- ($(VG)+(-0.5,0.4)$); 
\draw[double, -stealth, thick] ($(OM)+(-1.1,-0.3)$) -- ($(VG)+(1.2,0.4)$); 
\draw[double, -stealth, thick] ($(OM)+(-0.3,-0.5)$) -- ($(MAT)+(-0.3,0.5)$); 
\draw[double, -stealth, thick] ($(OM)+(1.1,0)$) -- ($(HOMEO)+(0,0.5)$); 
\draw[double, -stealth, thick] ($(HOMEO)+(0,-0.5)$) -- ($(HTP)+(0,0.5)$); 
\draw[double, -stealth, thick] ($(HTP)+(-0.1,-0.5)$) -- ($(COM2)+(-0.3,0.6)$); 
\draw[double, -stealth, thick] ($(MAT)+(0,-0.5)$) -- ($(OS)+(-0.5,0.5)$); 
\draw[double, -stealth, thick] ($(OS)+(-0.3,-0.5)$) -- ($(GROS)+(-0.3,0.5)$); 
\draw[double, -stealth, thick] ($(VG)+(-0.2,-0.5)$) -- ($(GRVG)+(0.3,0.5)$); 
\draw[double, -stealth, thick] ($(GRVG)+(0.3,-0.5)$) -- ($(GRAB)+(0.3,0.5)$); 
\draw[double, stealth-stealth, thick] ($(GRAB)+(1.55,0)$) -- ($(GROS)+(-1.5,0)$); 
\draw[double, stealth-stealth, thick] ($(GROS)+(1.4,0)$) -- ($(CHAR)+(-1.2,0)$);

\draw[double, -stealth, thick] ($(CIR)+(-1.1,0)$) -- ($(M3)+(0,0.55)$); 
\draw[double, -stealth, thick] ($(M3)+(0,-0.55)$) -- ($(COM3)+(0,0.6)$); 
\draw[double, stealth-stealth, thick] ($(COM3)+(1.1,0)$) -- ($(GRVG)+(-1.35,0)$); 
\draw[double, stealth-stealth, thick] ($(COM2)+(-1.1,0)$) -- ($(OS)+(1.35,0)$); 

\draw[double, -stealth, densely dashed, thick] ($(HOMEO)+(-1.1,0)$) -- node[anchor=south east] {\footnotesize $\ell=3$}($(MAT)+(1.05,0.35)$); 
\draw[double, stealth-stealth, densely dashed, thick] ($(GRVG)+(1.45,0.1)$) -- node[anchor=south] {\footnotesize $\cha R=2$} ($(OS)+(-1.35,0.1)$); 

\draw[-stealth, double, densely dashed, thick, blue] ($(VG)+(0.8,0.4)$) -- node[anchor=south east] {\footnotesize (\S \ref{subsec:gencodim2})} ($(OM)+(-1.1,0)$); 
\draw[-stealth, dotted, thick, blue] ($(MAT)+(-1.1,0)$) -- node {\footnotesize NO} node[anchor=north east] {\footnotesize (\S \ref{subsec:nonisom})} ($(VG)+(1.3,0)$); 
\draw[-stealth, dotted, thick, blue] ($(MAT)+(-1.1,-0.15)$) -- node[pos=0.25] {\footnotesize NO} node[anchor=south east, pos=0.25] {\footnotesize (\S \ref{subsec:signedcir})} ($(COM3)+(1.1,0.5)$); 
\draw[-stealth, dotted, thick, blue] ($(MAT)+(-1.1,-0.3)$) -- node {\footnotesize NO} node[anchor=north west] {\footnotesize (\S \ref{subsec:signedcir})} ($(GRVG)+(1.3,0.5)$); 
\draw[stealth-, dotted, thick, blue] ($(GRVG)+(1.45,-0.2)$) -- node {\footnotesize NO} node[anchor=north] {\footnotesize (\S \ref{subsec:signedcir})} ($(OS)+(-1.35,-0.2)$); 
\draw[-stealth, double, densely dashed, thick, blue] ($(GRVG)+(-1.2,0.5)$) -- node[anchor=north west] {\footnotesize (\S \ref{subsec:signedcir})}    ($(CIR)+(-0.7,-0.5)$);


\draw[-stealth, dotted, thick] ($(OS)+(0,0.5)$) -- node {\footnotesize NO}($(MAT)+(0.5,-0.5)$); 
\draw[-stealth, dotted, thick] ($(HTP)+(0.5,0.5)$) -- node {\footnotesize NO}($(HOMEO)+(0.5,-0.5)$); 
\draw[-stealth, dotted, thick] ($(MAT)+(1.05,0.2)$) -- node {\footnotesize NO}($(HTP)+(-1.2,0.2)$); 
\draw[-stealth, dotted, thick] ($(HTP)+(-1.2,-0.2)$) -- node {\footnotesize NO}($(MAT)+(1.05,-0.2)$); 
\draw[-stealth, dotted, thick] ($(MAT)+(0.3,0.5)$) -- node {\footnotesize NO}($(OM)+(0.3,-0.5)$); 
\draw[-stealth, dotted, thick] ($(GROS)+(0.3,0.5)$) --node {\footnotesize NO} ($(OS)+(0.3,-0.5)$); 
\draw[stealth-, dotted, thick] ($(HTP)+(+0.5,-0.5)$) -- node {\footnotesize NO} ($(COM2)+(0.3,0.5)$);

\draw[stealth-, dotted, thick, blue] ($(GRVG)+(-0.3,-0.5)$) -- node {\footnotesize NO} node[anchor=north east] {\footnotesize (\S\ref{subsec:sq0})
} ($(GRAB)+(-0.3,0.5)$); 


\draw (VG) node[fill=white, draw=black] 
{$\substack{\text{Filtered $R$-algebra}\\ \VG(\A_1)\simeq \VG(\A_2)}$}; 

\draw (GRVG) node[fill=white, draw=black] 
{$\substack{\text{Graded $R$-algebra}\\ \grVG^\bullet(\A_1)\simeq \grVG^\bullet(\A_2)}$}; 

\draw (GRAB) node[fill=white, draw=black] 
{$\substack{\text{Graded $R$-module}\\ \grVG^\bullet(\A_1)\simeq \grVG^\bullet(\A_2)}$}; 

\draw (CIR) node[fill=white, draw=black] 
{$\substack{\text{Signed circuits}\\ \circuit(\A_1)\simeq\circuit(\A_2)}$}; 

\draw (OM) node[fill=white, draw=black] 
{$\substack{\text{Tope graph}\\ \Tope(\A_1)\simeq\Tope(\A_2)}$}; 

\draw (MAT) node[fill=white, draw=black] 
{$\substack{\text{Lattice}\\ L(\A_1)\simeq L(\A_2)}$}; 

\draw (OS) node[fill=white, draw=black] 
{$\substack{\text{Graded $R$-algebra}\\ \OS^\bullet(\A_1)\simeq\OS^\bullet(\A_2)}$}; 

\draw (GROS) node[fill=white, draw=black] 
{$\substack{\text{Graded $R$-module}\\ \OS^\bullet(\A_1)\simeq\OS^\bullet(\A_2)}$}; 

\draw (HOMEO) node[fill=white, draw=black] 
{$\substack{\text{Homeomorphic}\\ M_2(\A_1)\simeq M_2(\A_2)}$}; 

\draw (HTP) node[fill=white, draw=black] 
{$\substack{\text{Homotopic}\\ M_2(\A_1)\simeq M_2(\A_2)}$}; 

\draw (CHAR) node[fill=white, draw=black] 
{$\substack{\text{Char. poly.}\\ \chi(\A_1, t)=\chi(\A_2,t)}$}; 

\draw (M3) node[fill=white, draw=black] 
{$\substack{\text{Homotopy equiv.}\\ \text{$3$-plexification}\\ M_3(\A)\simeq M_3(\A_2)}$}; 

\draw (COM3) node[fill=white, draw=black] 
{$\substack{\text{Cohom. ring}\\ H^\bullet(M_3(\A_1))\\ \simeq H^\bullet(M_3(\A_2))}$}; 

\draw (COM2) node[fill=white, draw=black] 
{$\substack{\text{Cohom. ring}\\ H^\bullet(M_2(\A_1))\\ \simeq H^\bullet(M_2(\A_2))}$}; 





\end{tikzpicture}
\caption{Solid doubled arrows are the known implications. 
Dashed doubled arrows represent conditional implications. 
Dotted arrows labeled ``NO'' indicate that counterexamples are known. 
Blue arrows correspond to implications addressed in this paper.}
\label{fig:implications}
\end{figure}

Finally, we will formulate a conjectural reconstruction algorithm 
for oriented matroids 
from VG algebras (Conjecture \ref{conj:algorithm}, Remark \ref{rem:reconstgr}). 

The paper is organized as follows. 
In \S \ref{sec:prel}, we fix notations and recall basic facts about 
hyperplane arrangements, VG algebras, and oriented matroids. 
As background, we also review the known implications 
among these notions (Figure \ref{fig:implications}). 

In \S \ref{sec:main}, we reconstruct tope graphs 
from filtered VG algebras. The key idea is that the set of 
chambers can be canonically identified with the set of primitive 
idempotent elements in $\VG(\A)$. 
A crucial observation is that the set of Heaviside functions 
can be recovered as the set of 
idempotent elements in $\Fil^1\VG(\A)$ (Lemma \ref{lem:equal}), 
under the assumption that $\A$ is generic in codimension $2$. 
We also exhibit a pair of generic arrangements that share 
the same intersection lattice but have non-isomorphic VG algebras. 

In \S \ref{sec:conj}, we analyze the set of 
idempotent elements in $\Fil^1\VG(\A)$, which we call 
generalized Heaviside functions (Theorem \ref{thm:genHeav}). 
Based on this, we formulate a non-deterministic algorithm that 
produces a graph from the filtered VG algebra. 
We conjecture that the algorithm always terminates 
with the tope graph. We illustrate the procedure 
for the $A_3$ arrangement in Example \ref{ex:nonHeav}. 

In \S \ref{sec:vgsign}, we discuss how to reconstruct 
oriented matroids from the graded VG algebra. First, 
we describe the set of square-zero elements in $\grVG^1(\A)$, 
which is closely related to generalized Heaviside 
functions. Then, again assuming that $\A$ is 
generic in codimension $2$, we show that the set of 
signed circuits can be recovered from the square-zero 
elements using the presentation of the graded VG algebra. 

\section{Preliminaries}
\label{sec:prel}

Unless otherwise stated, $R$ denotes an integral domain, and 
$\A$ always denotes a real central hyperplane arrangement. 

\subsection{Hyperplane arrangements and oriented matroids}
\label{subsec:hyparr}

Let $\A=\{H_1, \dots, H_n\}$ be a linear hyperplane arrangement 
in the real vector space $V=\R^\ell$. The complement 
$V\setminus\bigcup_{i=1}^nH_i$ decomposes into 
disjoint open cones, called \emph{chambers} 
(or \emph{topes} in the context of oriented matroids). 
We denote the set of all chambers by $\ch(\A)$. 

Choose a defining linear form $\alpha_i\in V^*$ for each 
hyperplane $H_i=\alpha_i^{-1}(0)$. 
The positive and negative half-spaces are denoted by 
$H_i^+=\alpha_i^{-1}(\R_{>0})$ and $H_i^-=\alpha_i^{-1}(\R_{<0})$, 
respectively. 
For chambers $C, C'\in\ch(\A)$, the set of separating hyperplanes 
is defined by 
\begin{equation}
\Sep(C, C')=\{
H\in\A\mid \text{$H$ separates $C$ and $C'$}\}. 
\end{equation}
The \emph{tope graph} $\Tope(\A)$ 
(or $\Tope(\alpha_1, \dots, \alpha_n)$) of $\A$ is 
the finite graph $(\ch(\A), E)$ with vertex set $\ch(\A)$ 
and edge set 
\begin{equation}
E=\{\{C, C'\}\mid \#\Sep(C, C')=1\}. 
\end{equation}
The graph distance between two vertices $C$ and $C'$ is 
given by $d(C, C')=\#\Sep(C, C')$. 

The arrangement $\A$ (or more precisely, the tuple of defining 
linear forms 
$\bm{\alpha}=(\alpha_1, \dots, \alpha_n)\in (V^*)^n$) determines an 
oriented matroid. 
One possible formulation is via the set of covectors 
\begin{equation}
\covec (\bm{\alpha})=
\{(\sgn \alpha_1(v), \dots, \sgn \alpha_n(v))
\in\{+, -, 0\}^n\mid v\in V\}, 
\end{equation}
where $\sgn: \R \to \{\pm, 0\}$ is the usual sign function, i.e., 
$\sgn(x) = +$ if $x > 0$, $\sgn(x) = -$ if $x < 0$, and 
$\sgn(x) = 0$ if $x = 0$. 

The sign vectors of the coefficients of linear dependence 
relations are called \emph{vectors}, and the set of vectors is 
denoted by 
\begin{equation}
\vecto (\bm{\alpha})=
\{
(\sgn \lambda_1, \dots, \sgn \lambda_n)
\in\{+, -, 0\}^n
\mid
\lambda_1\alpha_1+\dots+\lambda_n\alpha_n=0, (\lambda_i\in\R)
\}. 
\end{equation}
For $\bm{\sigma}\in\vecto(\bm{\alpha})$, 
define the support, positive support, and negative support by 
$\Supp(\bm{\sigma})=\{i\in[n]\mid\sigma_i\neq 0\}$, 
$\Supp^+(\bm{\sigma})=\{i\in[n]\mid\sigma_i=+\}$, and 
$\Supp^-(\bm{\sigma})=\{i\in[n]\mid\sigma_i=-\}$, respectively. 
A support-minimal dependent set is called a \emph{circuit}. 
The set of signed circuits is defined as 
\begin{equation}
\circuit(\bm{\alpha})=
\left\{
(\sgn \lambda_1, \dots, \sgn \lambda_n)
\middle| 
\ 
  \begin{lgathered} 
\lambda_1\alpha_1+\dots+\lambda_n\alpha_n=0 \mbox{ is a}\\
\mbox{minimal dependence relation}
  \end{lgathered} 
\right\}, 
\end{equation}
For a signed circuit 
$\bm{\sigma}=(\sigma_1, \dots, \sigma_n)\in\circuit(\bm{\alpha})$, 
one has 
\begin{equation}
\bigcap_{i\in\Supp^+(\bm{\sigma})}H_i^+\cap
\bigcap_{i\in\Supp^-(\bm{\sigma})}H_i^-=\emptyset.
\end{equation}

Note that for each hyperplane $H_i$, there are two possible 
choices of defining linear form, $\alpha_i$ and $-\alpha_i$. 
Changing the sign of $\alpha_i$ does not 
affect the hyperplane arrangement itself. We say that 
covectors $\covec (\alpha_1, \dots, \alpha_n)$ and 
$\covec (\beta_1, \dots, \beta_n)$ are 
\emph{reorientation equivalent} if there exist 
$\eps_1, \dots, \eps_n\in\{\pm 1\}$ such that 
\begin{equation}
\covec (\alpha_1, \dots, \alpha_n)=
\covec (\eps_1\beta_1, \dots, \eps_n\beta_n). 
\end{equation}
Two oriented matroids are said to be \emph{isomorphic} if 
they are equivalent under a reorientation and a permutation of 
the ground set $[n]=\{1, \dots, n\}$. 
If two oriented matroids are isomorphic, then 
clearly their tope graphs are isomorphic. The converse is also 
true. 
\begin{proposition}[\cite{ori-mat}, Theorem 4.2.14]
\label{prop:graphisom}
Two oriented matroids are isomorphic if and only if 
their tope graphs are isomorphic. 
\end{proposition}

\subsection{Intersection lattice and Orlik-Solomon algebra}
\label{subsec:OSalg}

See \cite{orl-ter} for details of this section. 

The set of intersections is defined by 
\begin{equation}
L(\A)=
\left\{
\left.
\bigcap_{H\in\mathcal{B}}H
\right|
\mathcal{B}\subset\A
\right\} 
\end{equation}
with a partial order given by 
$X\leq Y\Longleftrightarrow X\supseteq Y$ 
($X, Y\in L(\A)$), which is called the intersection lattice. 
The poset is called the \emph{intersection lattice}. 
The data of the intersection lattice $L(\A)$ is equivalent to 
the underlying matroid of $(\alpha_1, \dots, \alpha_n)$.  
Recall that the M\"obius function $\mu: L(\A)\to\Z$ is defined by 
$\mu(V)=1$ and $\mu(X)=-\sum_{V\leq Y<X}\mu(Y)$ for $V<X$. 
The \emph{characteristic polynomial} of $\A$ is then defined as 
$\chi(\A, t)=\sum_{X\in L(\A)}\mu(X)t^{\dim X}$. 
By the definition of M\"obius function, if $\A\neq\emptyset$, 
we have $\chi(\A, 1)=0$, equivalently, $(t-1)|\chi(\A, t)$. 

Let us recall the definition of the Orlik-Solomon algebra. 
Let $E=Re_1\oplus\cdots\oplus Re_n$ be a free $R$-module of 
rank $n$, and let $\wedge E=\bigoplus_{k=0}^n E^{\wedge k}$ be the 
exterior algebra. Define a linear map $\partial: E^{\wedge k}
\to E^{\wedge(k-1)}$ by 
\begin{equation}
\label{eq:OSrel}
\partial e_{i_1, \dots, i_k}=\sum_{p=1}^k(-1)^{p-1}
e_{i_1}\wedge\cdots\wedge\widehat{e_{i_p}}\wedge\cdots
\wedge e_{i_k}. 
\end{equation}
A subset $S=\{i_1, \dots, i_k\}\subset [n]$ is 
called \emph{dependent} if the intersection 
$H_S:=H_{i_1}\cap\cdots\cap H_{i_k}$ has codimension 
strictly smaller than $k$. 
The Orlik-Solomon algebra $\OS^\bullet_R(\A)$ is defined as 
the quotient algebra 
\begin{equation}
\OS_R^\bullet(\A)=
\frac{\wedge E}{\langle \partial e_S\mid S\subset 
[n]\mbox{ is dependent}\rangle}. 
\end{equation}
Recall that 
$M_2(\A)\simeq\C^\ell\setminus\bigcup_{H\in\A}H\otimes\C$ 
the complexified complement of $\A$. Orlik and Solomon proved 
that the cohomology ring $H^\bullet(M_2(\A), R)$ is isomorphic to 
$\OS^\bullet_R(\A)$ as $R$-algebras. 
When $R=\R$, considering de Rham presentation of 
$H^\bullet(M_2(\A), \R)$, 
the isomorphism is realized by the map 
$e_i\mapsto d\log\alpha_i/2\pi\sqrt{-1}$. 
The rank $b_k=\rank\OS^k(\A)$ is equal to the $k$-th Betti 
number of the space $M_2(\A$). It is also known that the coefficients 
of the characteristic polynomial are the Betti numbers with alternating 
sign. More precisely, 
\begin{equation}
\chi(\A, t)=t^\ell-b_1t^{\ell-1}+b_2t^{\ell-2}-\cdots+(-1)^\ell b_\ell. 
\end{equation}
We also note that the sum of Betti numbers is equal to the 
number of chambers: 
$\#\ch(\A)=(-1)^\ell\chi(\A, -1)=\sum_{k=0}^\ell b_k$ \cite{zas}. 

\subsection{Filtered and graded Varchenko-Gelfand algebras}
\label{subsec:vgalg}

The set of all $R$-valued functions on $\ch(\A)$ is called the 
\emph{Varchenko-Gelfand (VG) algebra}, denoted 
\[
\VG(\A)_R=\{f: \ch(\A)\longrightarrow R\}. 
\]
We sometimes omit the subscript $R$. 
As an $R$-algebra, $\VG(\A)_R$ is isomorphic to the direct 
product $R^{\ch(\A)}$. Recall that an element $f\in\VG(\A)$ 
is called \emph{idempotent} if $f^2=f$, and 
\emph{primitive idempotent} 
if it can not be written as the sum of two nonzero 
idempotent elements. 
For $C\in\ch(\A)$, 
let $1_C:\ch(\A)\to\R$ be the characteristic function, 
defined by $1_C(C') = \delta_{C, C'}$. 
Then the sets of idempotents and primitive idempotents are given by 
\begin{equation}
\label{eq:primid}
\begin{aligned}
\Idem(\VG(\A))=&\left\{\left. \sum_{C\in S}1_C\right| 
S\subset\ch(\A)\right\}\\
\PrimIdem(\VG(\A))=&\{1_C\mid C\in\ch(\A)\}. 
\end{aligned}
\end{equation}
Thus the set of primitive idempotents is canonically identified 
with $\ch(\A)$. 

\begin{definition}
For each hyperplane $H_i$, 
define $x_i^+:\ch(\A)\to R$ by 
\begin{equation}
\label{eq:01prod}
x_i^+(C)=
\begin{cases}
1, & (C\subset H_i^+)\\
0, & (C\not\subset H_i^+).
\end{cases}
\end{equation}
Similarly, define $x_i^-$ with respect to $H_i^-$. 
These are called the \emph{Heaviside functions}. 
They satisfy the relations 
\begin{equation}
x_i^++x_i^-=1, \ \ \ \ \ (x_i^\pm)^2=x_i^\pm, \ \ \ \ \ x_i^+\cdot x_i^-=0. 
\end{equation}
\end{definition}
The VG algebra $\VG(\A)$ is generated by 
Heaviside functions (see Proposition \ref{prop:vg1}). 
This yields a natural filtration according to degree 
in the Heaviside functions. 
\begin{definition}
Define the subspace $\Fil^k\VG(\A)\subset\VG(\A)$ by 
\begin{equation}
\Fil^k\VG(\A)=
\left\{
y\in\VG(\A)
\middle| \ 
\begin{lgathered} 
\exists 
G(z_1, \dots, z_n)\in R[z_1, \dots, z_n], \\
\deg G\leq k, \mbox{ and }G(x_1^+, \dots, x_n^+)=y
\end{lgathered} 
\right\}. 
\end{equation}
We set $\Fil^{-1}\VG(\A)=\{0\}$. 
Clearly, $\Fil^{0}\VG(\A)=R$. The increasing filtration 
$0=\Fil^{-1}\subset\Fil^0\subset\Fil^1\subset
\Fil^2\subset\cdots$ is called the 
\emph{VG filtration}. 
Equipped with this filtration, $\VG(\A)$ is sometimes 
referred to as the \emph{filtered VG algebra}.) 
\end{definition}
Since $\Fil^i\cdot\Fil^j\subset\Fil^{i+j}$, we may define 
\begin{equation}
\grVG^k(\A)=\Fil^k/\Fil^{k-1}, \ \ 
\grVG^\bullet(\A)=
\bigoplus_{k\geq 0}\grVG^k(\A), 
\end{equation}
which we call the \emph{graded VG algebra}. 

\begin{proposition}[Varchenko and Gelfand \cite{var-gel}]
\label{prop:vg1}\ 
\begin{itemize}
\item[(1)] 
$\Fil^\ell\VG(\A)=\VG(\A)$. 
\item[(2)] 
$\grVG^k(\A)\simeq\OS^k(\A)$ as $R$-modules. In particular, 
$\rank\grVG^k(\A)=b_k$. 
\item[(3)] 
If $R=\F_2$ (or more generally, if $\cha R=2$), 
then $\grVG^\bullet(\A)\simeq\OS^\bullet(\A)$ as $\F_2$-algebras. 
\end{itemize}
\end{proposition}

\begin{proposition}[\cite{var-gel}] 
\label{prop:presenVG}
Let $\mathcal{I}_\A\subset R[e_1, \dots, e_n]$ be the ideal generated 
by the following elements: 
\begin{itemize}
\item 
$e_i^2-e_i$, for $i=1, \dots, n$; 
\item 
for each signed circuit $\bm{\sigma}\in\circuit(\bm{\alpha})$, 
\begin{equation}
\label{eq:vgrel}
\prod\limits_{i\in\Supp^+(\bm{\sigma})}e_i\cdot
\prod\limits_{i\in\Supp^-(\bm{\sigma})}(e_i-1)
-
\prod\limits_{i\in\Supp^+(\bm{\sigma})}(e_i-1)\cdot
\prod\limits_{i\in\Supp^-(\bm{\sigma})}e_i. 
\end{equation}
\end{itemize}
Then the map $e_i\mapsto x_i^+$ gives an isomorphism of 
$R$-algebras $R[e_1, \dots, e_n]/\mathcal{I}_\A
\stackrel{\simeq}{\longrightarrow}
\VG(\A)$. 
\end{proposition}
Let $\bm{\sigma}\in\circuit(\bm{\alpha})$ be a signed circuit 
with $\Supp(\bm{\sigma})=\{i_1, \dots, i_k\}$. Then there exists a 
unique (up to scalar) nontrivial linear relation 
\begin{equation}
\label{eq:linearrelalpha}
\lambda_{i_1}\alpha_{i_1}+\dots+\lambda_{i_k}\alpha_{i_k}=0. 
\end{equation}
The leading term of \eqref{eq:vgrel} is a degree $k-1$ homogeneous 
polynomial 
\begin{equation}
\label{eq:grVG}
\sum_{p=1}^k\sgn(\lambda_{i_p})\cdot 
e_{\Supp(\bm{\sigma})\setminus\{i_p\}}, 
\end{equation}
where $e_{\Supp(\bm{\sigma})\setminus\{i_k\}}=
e_{i_1}\cdots\widehat{e_{i_p}}\cdots e_{i_k}$.
This relation also appears in the combinatorial analogue of 
the Orlik-Terao algebra \cite{cor-com}. 

\begin{proposition}[\cite{mos-equ}, Theorem 5.9] 
\label{prop:presenGRVG}
Let $I_\A\subset R[e_1, \dots, e_n]$ be the ideal generated 
by: 
\begin{itemize}
\item 
$e_i^2$, for $i=1, \dots, n$; 
\item 
for each signed circuit 
$\bm{\sigma}\in\circuit(\bm{\alpha})$ with 
$\Supp(\bm{\sigma})=\{i_1, \dots, i_k\}$ and 
linear relation \eqref{eq:linearrelalpha}, 
$\sum_{p=1}^k\sgn(\lambda_{i_p})\cdot 
e_{\Supp(\bm{\sigma})\setminus\{i_p\}}$ 
(as in \eqref{eq:grVG})  
\end{itemize}
Then the map $e_i\mapsto \overline{x_i^+}$ induces an isomorphism of 
graded $R$-algebras 
$R[e_1, \dots, e_n]/I_\A
\stackrel{\simeq}{\longrightarrow}
\grVG^\bullet(\A)$. 
\end{proposition}
Note that when $\cha R=2$, relation \eqref{eq:grVG} 
coincides with \eqref{eq:OSrel}, which 
explains the isomorphism in Proposition \ref{prop:vg1} (3). 

\begin{remark}
\label{rem:uniqurel}
Cordovil \cite[Corollary 2.8]{cor-com} proved that 
NBC sets form a basis of $\grVG^\bullet(\A)$. It follows 
that for a signed circuit $\bm{\sigma}\in\circuit(\bm{\alpha})$, 
the elements 
$e_{\Supp(\bm{\sigma})\setminus\{i_2\}}, \dots, 
e_{\Supp(\bm{\sigma})\setminus\{i_k\}}$ 
can be part of a basis of $\grVG^{k-1}(\A)$. 
Moreover: 
\begin{itemize}
\item 
$e_{\Supp(\bm{\sigma})\setminus\{i_1\}}, 
e_{\Supp(\bm{\sigma})\setminus\{i_2\}}, \dots, 
e_{\Supp(\bm{\sigma})\setminus\{i_k\}}$ form a 
circuit in $\grVG^{k-1}(\A)$.
\item 
The relation 
$\sum_{p=1}^k\sgn(\lambda_{i_p})\cdot
e_{\Supp(\bm{\sigma})\setminus\{i_p\}}=0$ is unique 
up to scalar multiple. 
\end{itemize}
\end{remark}

For the inductive arguments, the following results are also useful. 
Fix $H\in\A$. We consider the associated arrangements: 
the deletion $\A':=\A\setminus\{H\}$ and 
the restriction $\A''=\A^H$, where $\A^H$ denotes the induced 
arrangement in the ambient space $H\simeq\R^{\ell-1}$. 
$\VG(\A')$ is naturally embedded into $\VG(\A)$ as a subalgebra. 
Moreover, $\VG(\A)$ and $\VG(\A^H)$ are related as 
follows. 
Let $C\in\ch(\A^H)$ be a chamber of the restriction. Then 
there are exactly two chambers $C^+$ and $C^-$ of $\A$ 
such that 
\begin{itemize}
\item $C^+\subset H^+$ and $C^-\subset H^-$, and  
\item $C$ is contained in the closures of both $C^+$ and $C^-$. 
\end{itemize}
(Figure \ref{fig:rho}.) 
Using  $C^\pm$, define linear maps $\rho^\pm: \VG(\A)\to 
\VG(\A^H), f\mapsto \rho^\pm f$ by 
\begin{equation}
(\rho_H^+f)(C):=f(C^+), \ \ 
(\rho_H^-f)(C):=f(C^-), 
\end{equation}
and set $\rho_H:=\rho_H^+-\rho_H^-$. 

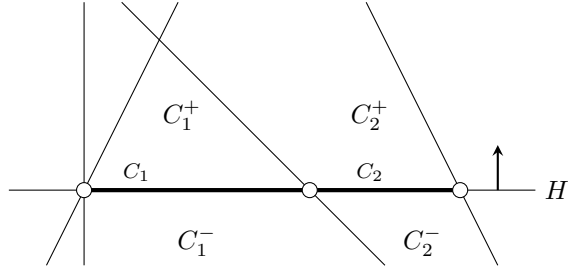
\begin{figure}[htbp]
\centering
\begin{tikzpicture}

\draw (0,1) -- (7,1) node[right] {$H$};
\draw[-stealth, thick] (6.5, 1) --++(0, 0.6);
\draw (0.5,0) -- ++(1.75,3.5);
\draw (1,0) -- ++(0,3.5);
\draw (5,0) -- ++(-3.5,3.5);
\draw (6.5,0) -- ++(-1.75,3.5);

\draw[ultra thick] (1,1) -- (6,1);
\draw (1.7,1) node[above] {\footnotesize $C_1$}; 
\draw (4.8,1) node[above] {\footnotesize $C_2$}; 

\draw (2.3,2) node {$C_1^+$}; 
\draw (2.5,0) node[above] {$C_1^-$}; 
\draw (4.8,2) node {$C_2^+$}; 
\draw (5.5,0) node[above] {$C_2^-$}; 

\filldraw[draw=black, fill=white] (1,1) circle [radius=0.1]; 
\filldraw[draw=black, fill=white] (4,1) circle [radius=0.1]; 
\filldraw[draw=black, fill=white] (6,1) circle [radius=0.1]; 

\end{tikzpicture}
\caption{$C_1, C_2\in\ch(\A^H)$ and their lifts 
$C_1^\pm, C_2^\pm\in\ch(\A)$.}
\label{fig:rho}
\end{figure}

Note that 
$(\rho_H f)(C)=0$ if and only if $f(C^+)=f(C^-)$. Hence 
$\rho_H f=0$ if and only if $f\in\VG(\A)$ descends to $f\in\VG(\A')$. 
We therefore have a short exact sequence 
\begin{equation}
0
\longrightarrow
\VG(\A')
\longrightarrow
\VG(\A)
\stackrel{\rho_H}{\longrightarrow}
\VG(\A^H). 
\end{equation}
Furthermore, Varchenko and Gelfand proved the following. 
\begin{proposition}[\cite{var-gel}]
\label{prop:vg2}
For any $H$ and $0\leq k\leq\ell$, the sequence 
\begin{equation}
\label{eq:shortexact}
0
\longrightarrow
\Fil^k\VG(\A')
\longrightarrow
\Fil^k\VG(\A)
\stackrel{\rho_H}{\longrightarrow}
\Fil^{k-1}\VG(\A^H)
\longrightarrow
0
\end{equation}
is exact. Moreover, $f\in\VG(\A)$ lies in 
$\Fil^k\VG(\A)$ if and only if 
$\rho_H(f)\in\Fil^{k-1}\VG(\A^H)$ for every $H\in\A$. 
\end{proposition}
By the snake lemma, \eqref{eq:shortexact} induces a short 
exact sequence of the graded version. 
\begin{equation}
0
\longrightarrow
\grVG^k(\A')
\longrightarrow
\grVG^k(\A)
\stackrel{\rho_H}{\longrightarrow}
\grVG^{k-1}(\A^H)
\longrightarrow
0
\end{equation}
The following are also useful and frequently used. 
\begin{proposition}
\label{prop:vgrho}
Let $f\in\VG(A)$. 
Then: 
\begin{itemize}
\item 
$f\in\Fil^0\VG(\A)$ if and only if 
$\rho_H(f)=0$ for all $H\in\A$; 
\item 
$f\in\Fil^1\VG(\A)$ if and only if 
$\rho_H(f)$ is a constant function on $H$ for all $H\in\A$. 
\end{itemize}
\end{proposition}

\begin{remark}
The correspondence between chambers and elements in 
Orlik-Solomon algebras (Proposition \ref{prop:vg1}, (2)) 
has appeared several times in recent works. 
In \cite{yos-ch, bai-yos}, a basis of  of $H^\ell(M_2(\A), \Z)$ 
corresponding to chambers (a ``chamber basis'') 
was constructed via the canonical isomorphism 
$[C]\in H^{\operatorname{BM}}_\ell(M_2(\A), \Z)
\stackrel{\simeq}{\longrightarrow} H^\ell(M_2(\A), \Z)$. 
This basis is a dual to a homology basis constructed 
by Morse-theoretic method \cite{yos-lef}. More recently, 
Eur and Lam \cite{eur-lam} extended these ideas to 
oriented matroids. In their work, the compatibility between 
the map $\rho$ and the residue map in Orlik-Solomon algebras 
plays a crucial role. 
\end{remark}

\subsection{Motivation and background}
\label{subsec:mot}

The isomorphism class of an oriented matroid determines 
not only the VG algebra and the intersection lattice, 
but also the 
topological space $M_2(\A)$ itself. The Salvetti complex 
\cite{sal-top} describes 
the homotopy type of $M_2$ in terms of oriented matroids. 
More generally, de Concini and Salvetti 
\cite{dec-sal} proved that the oriented matroid determines 
the homotopy type of the $c$-plexified complement $M_c(\A)$. 
Furthermore, Bj\"orner and Ziegler \cite{sal-top, bz-st} 
proved that the oriented matroid recovers 
the homeomorphism type of the complexified complement $M_2(\A)$. 
The relationship between topology of $M_2(\A)$ and the 
combinatorial structure of $L(\A)$ is one of the central 
topics in the theory of hyperplane arrangements. 
Many (non/partial)implications are known among 
these topological and combinatorial objects. 
Some representative results are as follows: 
\begin{itemize}
\item[$(1)$] 
When $\ell=3$, the homeomorphism type of $M_2(\A)$ 
determines $L(\A)$ \cite{jy-int}. 
\item[$(2)$]  
There exist arrangements $\A_1$ and $\A_2$ such that 
$M_2(\A_1)$ and $M_2(\A_2)$ are homotopy equivalent, 
but $L(\A_1)$ and $L(\A_2)$ are not isomorphic 
\cite[Example 3.1]{fal-hom}. 
Homotopy equivalence 
between $M_2(\A_1)$ and $M_2(\A_2)$ implies that their 
cohomology rings are isomorphic. 
Hence this example also shows that there exist 
$\A_1, \A_2$ with $\OS^\bullet(\A_1)\simeq\OS^\bullet(\A_2)$ 
as graded algebras, while $L(\A_1)\not\simeq L(\A_2)$. 
\item[$(3)$]  
The example in $(2)$ occurs in dimension $\ell=3$. 
Together with $(1)$, this implies that 
$M_2(\A_1)$ and $M_2(\A_2)$ are 
homotopy equivalent, but not homeomorphic. 
\item[$(4)$]  
There exist arrangements $\A_1$ and $\A_2$ such that 
$L(\A_1)\simeq L(\A_2)$ as posets, but 
$M_2(\A_1)$ and $M_2(\A_2)$ are not 
homotopy equivalent. Such a pair was first constructed 
over $\C$ by Rybnikov in \cite{ryb}, and later 
over $\R$ \cite{gue-sos, bgv-fundam}. 
\item[$(5)$] 
There exist $\A_1$ and $\A_2$ such that 
$\OS^\bullet(\A_1)$ and $\OS^\bullet(\A_2)$ are isomorphic 
as graded $R$-module 
(equivalently, $\chi(\A_1, t)=\chi(\A_2, t)$), but not isomorphic 
as graded algebras. See \cite[Example 3.1]{fal-alg}, 
\cite[Example 4.9, 4.10]{fal-coh}. 
\end{itemize}
These results are summarized in Figure \ref{fig:implications}. 
As noted in 
\S \ref{sec:intro}, $\grVG^\bullet(\A)$ is isomorphic to 
the cohomology ring $H^\bullet(M_3(\A), R)$ of the $3$-plexified complement $M_3(\A)$, 
and when $\cha R=2$, it is also 
isomorphic to $\OS^\bullet(\A)$ (Proposition \ref{prop:vg1}). 
One of the purposes of this paper is to understand 
the extent to which filtered and graded VG algebras 
reflect oriented matroid structures when $\cha R \neq 2$, 
in contrast to the case $\cha R = 2$.

\section{From filtered VG algebras to tope graphs}
\label{sec:main}

\subsection{Recovering from the set of Heaviside functions}
\label{subsec:heaviside}

Let $\A_1, \A_2$ be central real hyperplane arrangements. 
For the moment, we do not assume that they have the same rank. 
However, we can easily recover both the rank and the number of 
hyperplanes from the filtered algebra $\VG(\A_i)$, as follows. 

\begin{proposition}
\label{prop:recrank}
Let $\A_1$ and $\A_2$ be arrangements. 
If $\VG(\A_1)\simeq \VG(\A_2)$ 
(or $\grVG^\bullet(\A_1)\simeq\grVG^\bullet(\A_2)$), 
then $\#\A_1=\#\A_2$ and the two arrangements have the same rank. 
\end{proposition}

\begin{proof}
By Proposition \ref{prop:vg1}, 
$\rank\grVG^1(\A_i)=b_1=\#\A_i$. 
The rank of $\A_i$ is equal to the largest integer $k>0$ 
such that $\grVG^k(\A_i)\neq 0$. 
\end{proof}

To recover the oriented matroid, it is enough (by 
Proposition \ref{prop:graphisom}) to recover the 
tope graph $\Tope(\A)$. The vertex set of the tope 
graph is identified with $\PrimIdem(\VG(\A))$. 
Our next task is to recover the adjacency relations. 

\begin{definition}
Let $\A=\{H_1, \dots, H_n\}$ be a real arrangement. 
Denote the set of all Heaviside functions by 
\begin{equation}
\Heav(\A)=\{x_i^\pm\mid i=1, \dots, n\}. 
\end{equation}
We call the hyperplane $H_i$ the \emph{support} of the 
Heaviside function $x_i^\pm$. 
(See Definition \ref{def:support} for the general notion of support.) 
\end{definition}
Note that $x, x'\in\Heav$ have the same support 
if and only if either $x=x'$ or $x+x'=1$. 
The following observation is straightforward. 

\begin{proposition}
\label{prop:rectopeg}
Let $\A$ and $\Heav(\A)$ be as above. 
For two chambers $C, C'\in\ch(\A)$, we have 
\begin{equation}
\#\left\{x\in\Heav(\A)\mid x\cdot 1_C \neq x\cdot 1_{C'} \text{ and }
0\in\{x\cdot 1_C, x\cdot 1_{C'}\}\right\}=
2\cdot\#\Sep(C, C'). 
\end{equation}
\end{proposition}

\begin{proof}
By definition of Heaviside functions, 
chambers $C$ and $C'$ are separated by $H_i$ 
if and only if 
one of $x_i^\pm\cdot 1_C$ and $x_i^\pm\cdot 1_{C'}$ 
is zero, and the other one is nonzero. 
The result follows immediately. 
\end{proof}
By Proposition \ref{prop:rectopeg}, $C$ and $C'$ are adjacent if and only if 
$\#\{x\in\Heav(\A)\mid x\cdot 1_C\neq x\cdot 1_{C'}, 
\text{ and }
0\in\{x\cdot 1_C, x\cdot 1_{C'}\}\}=2$. 
Therefore, once 
the set of Heaviside functions $\Heav(\A)\subset\VG(\A)$ 
is known, the adjacency relation on the vertex set 
$\PrimIdem(\VG(\A))$, hence the tope graph $\Tope(\A)$, 
can be recovered algebraically. 

The main difficulty of the reconstruction problem is that, 
a priori, the set of Heaviside functions $\Heav(\A)$ is not known. 
From the definitions, we have 
$\Heav(\A)\subseteq\Fil^1\VG(\A)$, 
$\Heav(\A)\subseteq\Idem(\VG(\A))$ and 
$1\notin\Heav(\A)$. 
Hence, 
\begin{equation}
\label{eq:heavF1idem}
\Heav(\A)\subseteq\Fil^1\cap\Idem(\VG(\A))\setminus\{0, 1\}. 
\end{equation}
Since the right-hand-side of \eqref{eq:heavF1idem} is 
described purely in terms of the filtered algebra $\VG(\A)$, 
if equality holds in \eqref{eq:heavF1idem} then the 
tope graph can be recovered from $\VG(\A)$. 
In general, however, the inclusion is strict, 
as the following example shows. 

\begin{example}
\label{ex:alternating}
Let $\A=\{H_1, H_2, H_3\}$ be lines in $\R^2$ as in 
Figure \ref{fig:F1idem} (upper left). 
Define 
$y^\pm\in\VG(\A)$ 
as in Figure \ref{fig:F1idem} (upper right). 
Then, 
\begin{equation}
y^+=x_1^++x_2^++x_3^+-1. 
\end{equation}
The left-hand side ($y^+$) is clearly contained in $\Idem(\VG(\A))$. 
The right-hand side lies in $\Fil^1\VG(\A)$. 
(We can also verify this fact by checking $\rho_{H_i}(y^\pm)$ is a 
constant function for any $i=1, 2, 3$ (Proposition \ref{prop:vg2}).) 
Thus $y\in\Fil^1\cap\Idem(\VG(\A))\setminus\{0, 1\}$, 
yet $y$ is not a Heaviside function. 
Hence $\Heav(\A)\subsetneqq 
\Fil^1\cap\Idem(\VG(\A))\setminus\{0, 1\}$. 
(We will see in Theorem \ref{thm:genHeav} that 
$\Fil^1\cap\Idem(\VG(\A))\setminus\{0, 1\}=
\{x_1^\pm, x_2^\pm, x_3^\pm, y^\pm\}$ .) 

In this situation, we can not algebraically distinguish 
$x_1^\pm, x_2^\pm, x_3^\pm$ from $y^\pm$. 
Indeed, the map 
$x_1^+\mapsto y^-, x_2^+\mapsto x_2^+, x_3^+\mapsto x_3^+$ 
defines an automorphism of the filtered algebra $\VG(\A)$. 
In this example, the automorphism group 
$\Aut_{\operatorname{\mathsf{filt}}}(\VG(\A))$ 
(see \S \ref{subsec:autom}), 
of filtered algebra is isomorphic to the wreath product 
$\{\pm 1\}\wr\mathfrak{S}_3\simeq
\{\pm 1\}^3\rtimes\mathfrak{S}_3$, 
a group of order $48$, which 
is strictly larger than the automorphism group of the tope graph 
(a group of order $12$). 
\begin{figure}[htbp]
\centering
\begin{tikzpicture}

\coordinate (A0) at (0,4); 
\coordinate (A7) at (2.7,4); 
\coordinate (A6) at (6,4); 
\coordinate (A8) at (9.5,4); 
\coordinate (A1) at (0,1); 
\coordinate (A2) at (2.5,1); 
\coordinate (A3) at (5,1); 
\coordinate (A4) at (7.5,1); 
\coordinate (A5) at (10,1); 

\draw [very thick] (A0) -- ++(1,0);
\draw [very thick] (A0) -- ++(0.5,1);
\draw [very thick] (A0) -- ++(-0.5,1);
\draw [very thick] (A0) -- ++(-1,0);
\draw [very thick] (A0) -- ++(-0.5,-1);
\draw [very thick] (A0) -- ++(0.5,-1);
\draw[->] ($(A0)+(0.9,0)$) node[right] {$H_1$} --++(0,0.3); 
\draw[->] ($(A0)+(-0.9,0)$)  --++(0,0.3); 
\draw[->] ($(A0)+(0.45,0.9)$) node[above] {$H_2$} --++(0.26,-0.13); 
\draw[->] ($(A0)+(-0.45,-0.9)$) --++(0.26,-0.13); 
\draw[->] ($(A0)+(-0.45,0.9)$) node[above] {$H_3$} --++(-0.26,-0.13); 
\draw[->] ($(A0)+(0.45,-0.9)$) --++(-0.26,-0.13); 

\draw (A7) node {$\Tope(\A)$}; 
\filldraw[draw=black, fill=black] (A7)+(0.75,0.45) circle [radius=0.1]; 
\filldraw[draw=black, fill=black] (A7)+(0,0.9) circle [radius=0.1]; 
\filldraw[draw=black, fill=black] (A7)+(-0.75,0.45) circle [radius=0.1]; 
\filldraw[draw=black, fill=black] (A7)+(-0.75,-0.45) circle [radius=0.1]; 
\filldraw[draw=black, fill=black] (A7)+(0,-0.9) circle [radius=0.1]; 
\filldraw[draw=black, fill=black] (A7)+(0.75,-0.45) circle [radius=0.1]; 

\draw[very thick] ($(A7)+(0.75,0.45)$)--($(A7)+(0,0.9)$)--($(A7)+(-0.75,0.45)$)--($(A7)+(-0.75,-0.45)$)--($(A7)+(-0.75,-0.45)$)--($(A7)+(0,-0.9)$)--($(A7)+(0.75,-0.45)$)--cycle;


\draw [very thick] (A6) -- ++(1,0);
\draw [very thick] (A6) -- ++(0.5,1);
\draw [very thick] (A6) -- ++(-0.5,1);
\draw [very thick] (A6) -- ++(-1,0) node[left] {$y^+:=$};
\draw [very thick] (A6) -- ++(-0.5,-1);
\draw [very thick] (A6) -- ++(0.5,-1);
\draw [very thick] (A6)+(0.5,0.3) node {$1$};
\draw [very thick] (A6)+(0,0.6) node {$0$};
\draw [very thick] (A6)+(-0.5,0.3) node {$1$};
\draw [very thick] (A6)+(-0.5,-0.3) node {$0$};
\draw [very thick] (A6)+(0,-0.6) node {$1$};
\draw [very thick] (A6)+(0.5,-0.3) node {$0$};

\draw [very thick] (A8) -- ++(1,0);
\draw [very thick] (A8) -- ++(0.5,1);
\draw [very thick] (A8) -- ++(-0.5,1);
\draw [very thick] (A8) -- ++(-1,0) node[left] {$y^-:=$};
\draw [very thick] (A8) -- ++(-0.5,-1);
\draw [very thick] (A8) -- ++(0.5,-1);
\draw [very thick] (A8)+(0.5,0.3) node {$0$};
\draw [very thick] (A8)+(0,0.6) node {$1$};
\draw [very thick] (A8)+(-0.5,0.3) node {$0$};
\draw [very thick] (A8)+(-0.5,-0.3) node {$1$};
\draw [very thick] (A8)+(0,-0.6) node {$0$};
\draw [very thick] (A8)+(0.5,-0.3) node {$1$};

\draw [very thick] (A1) -- ++(1,0);
\draw [very thick] (A1) -- ++(0.5,1);
\draw [very thick] (A1) -- ++(-0.5,1);
\draw [very thick] (A1) -- ++(-1,0);
\draw [very thick] (A1) -- ++(-0.5,-1);
\draw [very thick] (A1) -- ++(0.5,-1);
\draw [very thick] (A1)+(0.5,0.3) node {$1$};
\draw [very thick] (A1)+(0,0.6) node {$0$};
\draw [very thick] (A1)+(-0.5,0.3) node {$1$};
\draw [very thick] (A1)+(-0.5,-0.3) node {$0$};
\draw [very thick] (A1)+(0,-0.6) node {$1$};
\draw [very thick] (A1)+(0.5,-0.3) node {$0$};

\draw (A1)+(1.25,0) node {$=$}; 

\draw [very thick] (A2) -- ++(1,0);
\draw [very thick] (A2) -- ++(0.5,1);
\draw [very thick] (A2) -- ++(-0.5,1);
\draw [very thick] (A2) -- ++(-1,0);
\draw [very thick] (A2) -- ++(-0.5,-1);
\draw [very thick] (A2) -- ++(0.5,-1);
\draw [very thick] (A2)+(0.5,0.3) node {$1$};
\draw [very thick] (A2)+(0,0.6) node {$1$};
\draw [very thick] (A2)+(-0.5,0.3) node {$1$};
\draw [very thick] (A2)+(-0.5,-0.3) node {$0$};
\draw [very thick] (A2)+(0,-0.6) node {$0$};
\draw [very thick] (A2)+(0.5,-0.3) node {$0$};

\draw (A2)+(1.25,0) node {$+$}; 

\draw [very thick] (A3) -- ++(1,0);
\draw [very thick] (A3) -- ++(0.5,1);
\draw [very thick] (A3) -- ++(-0.5,1);
\draw [very thick] (A3) -- ++(-1,0);
\draw [very thick] (A3) -- ++(-0.5,-1);
\draw [very thick] (A3) -- ++(0.5,-1);
\draw [very thick] (A3)+(0.5,0.3) node {$1$};
\draw [very thick] (A3)+(0,0.6) node {$0$};
\draw [very thick] (A3)+(-0.5,0.3) node {$0$};
\draw [very thick] (A3)+(-0.5,-0.3) node {$0$};
\draw [very thick] (A3)+(0,-0.6) node {$1$};
\draw [very thick] (A3)+(0.5,-0.3) node {$1$};

\draw (A3)+(1.25,0) node {$+$}; 

\draw [very thick] (A4) -- ++(1,0);
\draw [very thick] (A4) -- ++(0.5,1);
\draw [very thick] (A4) -- ++(-0.5,1);
\draw [very thick] (A4) -- ++(-1,0);
\draw [very thick] (A4) -- ++(-0.5,-1);
\draw [very thick] (A4) -- ++(0.5,-1);
\draw [very thick] (A4)+(0.5,0.3) node {$0$};
\draw [very thick] (A4)+(0,0.6) node {$0$};
\draw [very thick] (A4)+(-0.5,0.3) node {$1$};
\draw [very thick] (A4)+(-0.5,-0.3) node {$1$};
\draw [very thick] (A4)+(0,-0.6) node {$1$};
\draw [very thick] (A4)+(0.5,-0.3) node {$0$};

\draw (A4)+(1.25,0) node {$-$}; 

\draw [very thick] (A5) -- ++(1,0);
\draw [very thick] (A5) -- ++(0.5,1);
\draw [very thick] (A5) -- ++(-0.5,1);
\draw [very thick] (A5) -- ++(-1,0);
\draw [very thick] (A5) -- ++(-0.5,-1);
\draw [very thick] (A5) -- ++(0.5,-1);
\draw [very thick] (A5)+(0.5,0.3) node {$1$};
\draw [very thick] (A5)+(0,0.6) node {$1$};
\draw [very thick] (A5)+(-0.5,0.3) node {$1$};
\draw [very thick] (A5)+(-0.5,-0.3) node {$1$};
\draw [very thick] (A5)+(0,-0.6) node {$1$};
\draw [very thick] (A5)+(0.5,-0.3) node {$1$};

\end{tikzpicture}
\caption{An element $y\in\Fil^1\cap\Idem(\VG(\A))$ that is 
not a Heaviside function. (Arrows indicate positive sides).}
\label{fig:F1idem}
\end{figure}
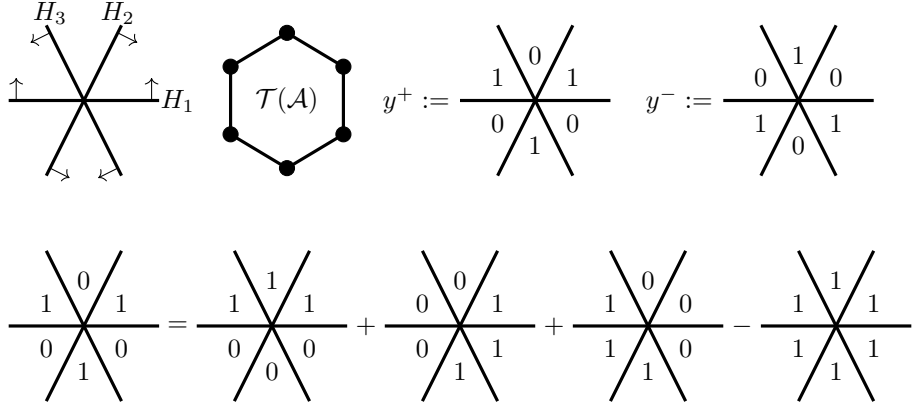
\end{example}

\subsection{Arrangements generic in codimension $2$}
\label{subsec:gencodim2}

\begin{definition}
\label{def:gencodim2}
An arrangement $\A$ is called \emph{generic in codimension $2$} if 
$\#\A_X=2$ for all $X\in L(\A)$ with $\codim X=2$. 
\end{definition}

Our first main result shows that 
the tope graph $\Tope(\A)$ can be recovered from 
the filtered algebra $\VG(\A)$ when 
$\A$ is generic in codimension $2$. 
\begin{theorem}
\label{thm:gencodim2}
Assume $\cha R\neq 2$. 
Let $\A_1, \A_2$ be real arrangements and 
suppose that $\A_1$ is generic in codimension $2$. 
If $\VG(\A_1)_R\simeq \VG(\A_2)_R$ as filtered algebras, 
then the tope graphs 
$\Tope(\A_1)$ and $\Tope(\A_2)$ are isomorphic. 
\end{theorem}

Before proving Theorem \ref{thm:gencodim2}, 
we introduce the notion of \emph{support} for elements 
of $\Fil^1\VG(\A)$. Recall that $y\in\Fil^1\VG(\A)$ 
if and only if $\rho_{H_i}(y)$ is a constant function on $H_i$ 
for every $i=1, \dots, n$ 
(Proposition \ref{prop:vg2}). 
\begin{definition}
\label{def:support}
For $y\in\Fil^1\VG(\A)$, define the set $\Supp(y)\subset\A$ by 
\begin{equation}
\Supp(y)=\{H\in\A\mid\rho_H(y)\neq 0\}. 
\end{equation}
\end{definition}

For $y\in\Fil^1$ the following are immediate: 
\begin{itemize}
\item 
$y\in\Fil^0\VG(\A)$ if and only if $\Supp(y)=\emptyset$. 
\item 
$y\in\Heav(\A)$ if and only if $\#\Supp(y)=1$.
\item 
For the element $y^\pm$ in Example \ref{ex:alternating}, we have 
$\Supp(y^\pm)=\A$. 
\end{itemize}

Now we come to the essential step in the proof 
of Theorem \ref{thm:gencodim2}. 

\begin{lemma}
\label{lem:equal}
Assume $\cha R\neq 2$. 
Let $\A$ be an arrangement generic in codimension $2$. 
Then $\Heav(\A)=\Fil^1\cap\Idem(\VG(\A))\setminus\{0, 1\}$. 
\end{lemma}
\begin{remark}
We will prove a more general result in Theorem \ref{thm:genHeav} 
without assuming ``generic in codimension $2$'' from which 
Lemma \ref{lem:equal} is immediate. 
\end{remark}
\begin{proof}[Proof of Lemma \ref{lem:equal}]
Let $y\in\Fil^1\cap\Idem(\VG(\A))\setminus\{0, 1\}$. 
Since $y\neq 0, 1$, we have $\#\Supp(y)>0$. 
Assume $\#\Supp(y)\neq 1$, i.e., $\#\Supp(y)\geq 2$, and 
choose distinct $H_1, H_2\in\Supp(y)$. 
Since $\A$ is generic in codimension $2$, 
the only hyperplanes containing a generic point 
$p\in H_1\cap H_2$ are $H_1$ and $H_2$, 
so a neighborhood of $p$ is divided into four chambers 
$C_1, \dots, C_4$ as in Figure \ref{fig:aroundp}. 
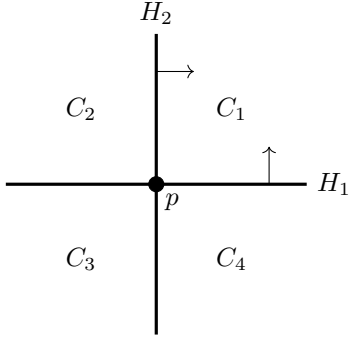
\begin{figure}[htbp]
\centering
\begin{tikzpicture}

\draw[very thick] (-2,0)--(2,0) node[right] {$H_1$}; 
\draw[very thick] (0,-2)--(0,2) node[above] {$H_2$}; 
\filldraw[draw=black, fill=black] (0,0) node[anchor=north west] {$p$} circle [radius=0.1]; 

\draw (1,1) node {$C_1$}; 
\draw (-1,1) node {$C_2$}; 
\draw (-1,-1) node {$C_3$}; 
\draw (1,-1) node {$C_4$}; 

\draw[->] (1.5,0)--++(0, 0.5);
\draw[->] (0,1.5)--++(0.5,0);

\end{tikzpicture}
\caption{For chambers $C_1, \dots, C_4$ around a generic point 
$p\in H_1\cap H_2$.}
\label{fig:aroundp}
\end{figure}
Because $y$ is idempotent, each $y(C_i)$ is either $0$ or $1$. 
Since $\rho_{H_1}(y)$ is constant, we have 
\begin{equation}
\rho_{H_1}(y)=y(C_1)-y(C_4)=y(C_2)-y(C_3)\neq 0. 
\end{equation}
Since $\cha R\neq 2$, this forces either\footnote{If 
$\cha R=2$, there are two more possibilities.} 
$y(C_1)=y(C_2)=1, y(C_3)=y(C_4)=0$ or 
$y(C_1)=y(C_2)=0, y(C_3)=y(C_4)=1$. 
In either case, 
$\rho_{H_2}(y)=0$, contradicting $H_2\in\Supp(y)$. 
Hence, $\#\Supp(y)=1$ and $y$ is a Heaviside function. 
\end{proof}
\begin{proof}[Proof of Theorem \ref{thm:gencodim2}]
By Lemma \ref{lem:equal}, 
\begin{equation}
\begin{aligned}
\#\Heav(\A_1)
&=\#\Fil^1\cap\Idem(\VG(\A_1))\setminus\{0, 1\}\\
&=\#\Fil^1\cap\Idem(\VG(\A_2))\setminus\{0, 1\}\\
&\geq\#\Heav(\A_2). 
\end{aligned}
\end{equation}
By Proposition \ref{prop:recrank}, $\#\Heav(\A_1)=\#\Heav(\A_2)$, 
so 
$\Fil^1\cap\Idem(\VG(\A_2))\setminus\{0, 1\}=\Heav(\A_2)$. 
Therefore the adjacency relations in the tope graphs 
(given by Proposition \ref{prop:rectopeg}) agree, so 
$\Tope(\A_1)\simeq\Tope(\A_2)$. 
\end{proof}

\subsection{Automorphism groups}
\label{subsec:autom}
We first consider the group of algebra automorphisms 
$\Aut_{\operatorname{\mathsf{alg}}}(\VG(\A))$ (not 
necessarily preserving the filtration). 
Any algebra isomorphism $f:\VG(\A)\to\VG(\A)$ 
preserves $\PrimIdem(\A)\simeq\ch(\A)$, hence 
induces a permutation of chambers 
$\overline{f}:\ch(\A)\stackrel{\simeq}{\to}\ch(\A)$. 
Conversely, any permutation 
$\overline{f}:\ch(\A)\stackrel{\simeq}{\to}\ch(\A)$ 
induces an algebra isomorphism 
$f:\VG(\A)\stackrel{\simeq}{\to}\VG(\A)$. 
Thus 
\begin{equation}
\Aut_{\operatorname{\mathsf{alg}}}(\VG(\A))\simeq
\Aut_{\operatorname{\mathsf{set}}}(\ch(\A)), 
\end{equation}
which is a group of order $\#\ch(\A) !$. 

We next compare the automorphism group of the filtered algebra 
$\Aut_{\operatorname{\mathsf{filt}}}(\VG(\A))$
with graph automorphism group 
$\Aut_{\operatorname{\mathsf{graph}}}(\Tope(\A))$ and 
the permutation group of chambers 
$\Aut_{\operatorname{\mathsf{set}}}(\ch(\A))$. 

\begin{theorem}
Assume $\cha R\neq 2$. 
Let $\A$ be a real arrangement in $\R^\ell$. 
\begin{itemize}
\item[$(1)$] 
There are natural inclusions 
\begin{equation}
\Aut_{\operatorname{\mathsf{graph}}}(\Tope(\A))
\subseteq 
\Aut_{\operatorname{\mathsf{filt}}}(\VG(\A))\subseteq 
\Aut_{\operatorname{\mathsf{set}}}(\ch(\A)). 
\end{equation}
In particular, $\Aut_{\operatorname{\mathsf{filt}}}(\VG(\A))$ 
is finite. 
\item[$(2)$] 
If $\A$ is generic in codimension $2$, then 
\begin{equation}
\Aut_{\operatorname{\mathsf{graph}}}(\Tope(\A))
=
\Aut_{\operatorname{\mathsf{filt}}}(\VG(\A)). 
\end{equation}
\end{itemize}
\end{theorem}
\begin{proof}
$(1)$ Since 
$\Aut_{\operatorname{\mathsf{filt}}}(\VG(\A))\subseteq 
\Aut_{\operatorname{\mathsf{alg}}}(\ch(\A))$, 
the inclusion 
$\Aut_{\operatorname{\mathsf{filt}}}(\VG(\A))\subseteq 
\Aut_{\operatorname{\mathsf{set}}}(\ch(\A))$ 
was observed above. 

If 
$f:\Tope(\A)\stackrel{\simeq}{\to}\Tope(\A)$ is a graph 
automorphism, it induces a permutation of chambers and 
hence an algebra automorphism of $\VG(\A)$. 
It remains to show that $f$ preserves the filtration, 
equivalently, $f$ preserves Heaviside functions. 

If $C_0, C_1\in\ch(\A)$ are adjacent and 
separated by a hyperplane $H$ with $C_0\subset H^+$ and 
$C_1\subset H^-$, then 
\begin{equation}
\label{eq:grahalf}
\begin{aligned}
\{C\in\ch(\A)\mid C\subset H^+\}&=\{C\in\ch(\A)\mid d(C, C_0)<d(C, C_1)\}\\
\{C\in\ch(\A)\mid C\subset H^-\}&=\{C\in\ch(\A)\mid d(C, C_0)>d(C, C_1)\}. 
\end{aligned}
\end{equation}
(Note that $d(C, C_0)=d(C, C_1)\pm 1$.) 
The right-hand sides of \eqref{eq:grahalf} are 
expressed purely in graph-theoretic terms, hence are 
preserved by graph automorphisms. 
Thus the Heaviside functions are preserved, so 
$f\in\Aut_{\mathsf{filt}}(\VG(\A))$. 

$(2)$ Let $f: \VG(\A)\stackrel{\simeq}{\to}\VG(\A)$ be a 
filtered automorphism. Then $f$ preserves 
$\Fil^1\VG(\A)$ and $\Idem(\A)$, and 
by Lemma \ref{lem:equal}, it preserves $\Heav(\A)$. 
By Proposition \ref{prop:rectopeg},  
adjacency of the chambers is determined by $\Heav(\A)$, 
so $f$ induces a graph automorphism. Hence 
$\Aut_{\mathsf{filt}}(\VG(\A))\subseteq 
\Aut_{\mathsf{graph}}(\Tope(\A))$, 
proving equality. 
\end{proof}

Example \ref{ex:alternating} shows that the strict inclusion 
\begin{equation}
\Aut_{\operatorname{\mathsf{graph}}}(\Tope(\A))
\subsetneq
\Aut_{\operatorname{\mathsf{filt}}}(\VG(\A)). 
\end{equation}
can occur in general when $\A$ is not generic in codimension $2$. 

\subsection{The intersection lattice does not determine the filtered VG algebra}
\label{subsec:nonisom}

In this section, we exhibit an explicit pair of arrangements 
$\A_1$ and $\A_2$ such that $L(\A_1)\simeq L(\A_2)$ as 
lattices but $\VG(\A_1)\not\simeq\VG(\A_2)$ as filtered 
algebras. Thus the intersection lattice does not determine 
the filtered Varchenko-Gelfand algebra. 

Based on Theorem \ref{thm:gencodim2}, it is enough 
to construct a pair of arrangements such that 
\begin{itemize}
\item they are generic in codimension $2$, 
\item their intersection lattices are isomorphic, 
\item their tope graphs are not isomorphic. 
\end{itemize}
We can construct such a (well-known) pair using 
generic $6$-planes in $\R^3$. 

\begin{example}
\label{ex:6planes}
Assume $\cha R\neq 2$. 
Let $\A_1$ and $\A_2$ be six generic planes in $\R^3$ as in 
Figure \ref{fig:6planes}. 

The intersection lattices $L(\A_1)$ and $L(\A_2)$ are isomorphic. 
However, tope graphs $\Tope(\A_1)$ and $\Tope(\A_2)$ are 
not isomorphic. 
Indeed, $\Tope(\A_1)$ has two vertices of degree $6$, 
whereas, $\Tope(\A_2)$ has none. 

These arrangements are 
clearly generic in codimension $2$. Therefore, 
$\VG(\A_1)\not\simeq\VG(\A_2)$ as filtered algebras. 
\begin{figure}[htbp]
\centering
\begin{tikzpicture}

\draw (2,0) node[below] {$\A_1$}; 
\draw[thick] (1.5,0) -- (1.5,3); 
\draw[thick] (1.5,3) .. controls (1.5,3.5) and (1.5,3.5) .. (2.25,4); 
\draw[thick] (2.25,0) -- (2.25,3);
\draw[thick] (2.25,3) .. controls (2.25,3.5) and (2.25,3.5) .. (1.5,4); 

\draw[thick] (0,0) -- (4,2);
\draw[thick] (4,2) .. controls (4.5,2.25) and (4.5,2.25) .. (5,3.25); 
\draw[thick] (0,0.75) -- (4,2.75);
\draw[thick] (4,2.75) .. controls (4.5,3) and (4.5,3) .. (5,2.5); 

\draw[thick] (4,0) -- (0,2);
\draw[thick] (4,0) .. controls (4.5,-0.25) and (4.5,-0.25) .. (5,0.75); 
\draw[thick] (4,0.75) -- (0,2.75);
\draw[thick] (4,0.75) .. controls (4.5,0.5) and (4.5,0.5) .. (5,-0.5); 

\draw (8,0) node[below] {$\A_2$}; 
\draw[thick] (7.5,0) -- (7.5,3); 
\draw[thick] (7.5,3) .. controls (7.5,3.5) and (7.5,3.5) .. (8.25,4); 
\draw[thick] (8.25,0) -- (8.25,3);
\draw[thick] (8.25,3) .. controls (8.25,3.5) and (8.25,3.5) .. (7.5,4); 

\draw[thick] (6,0) -- (10,2);
\draw[thick] (10,2) .. controls (10.5,2.25) and (10.5,2.25) .. (11,3.25); 
\draw[thick] (6,0.75) -- (10,2.75);
\draw[thick] (10,2.75) .. controls (10.5,3) and (10.5,3) .. (11,2.5); 

\draw[thick] (10,0) -- (6,2);
\draw[thick] (10,0) .. controls (10.5,-0.25) and (10.5,-0.25) .. (11,0.75); 
\draw[thick] (10,1.25) -- (6,3.25);
\draw[thick] (10,1.25) .. controls (10.5,1) and (10.5,1) .. (11,-0.5); 

\end{tikzpicture}
\caption{Six generic planes in $\R^3$.}
\label{fig:6planes}
\end{figure}
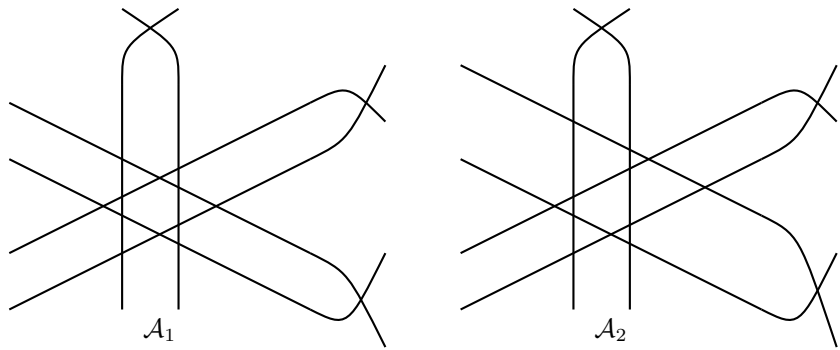
\end{example}

\section{Conjectural reconstruction algorithm}
\label{sec:conj}

\subsection{The Sylvester-Gallai Theorem and generalized Heaviside functions}
\label{subsec:gheav}

For simplicity, let 
\begin{equation}
\gHeav(\A)=\Fil^1\cap\Idem(\VG(\A))\setminus\{0, 1\}, 
\end{equation}
and call an element $y\in\gHeav(\A)$ a 
\emph{generalized Heaviside function}. 
The key step in \S \ref{subsec:gencodim2} (Lemma \ref{lem:equal}) 
was that if $\A$ is generic in codimension $2$, then 
\begin{equation}
\Heav(\A)=\gHeav(\A). 
\end{equation}
In this section, we analyze $\gHeav(\A)$ more closely. 
The element in Example \ref{ex:alternating} is a generalized Heaviside 
function. We can generalize this construction as follows. 

\begin{example}
\label{eq:stripe}
Let $H_1, H_2, \dots, H_{2s+1}$ be $2s+1$ ($s\in\Z_{\geq 0}$) 
hyperplanes such that $X=\bigcap_{i=1}^{2s+1}H_i$ is codimension $2$. 
We define the generalized Heaviside functions 
$\Alt^\pm (H_1, \dots, H_{2s+1})$ as follows. 

(Intuitive definition) 
There are $4s+2$ chambers. Assign values $0$ and $1$ 
alternatingly to these chambers. This produces two 
idempotent elements $\Alt^+(H_1, \dots, H_{2s+1})$ 
and $\Alt^-(H_1, \dots, H_{2s+1})$ in $\gHeav(\A)$ 
(Figure \ref{fig:Alt}). 
Note that the case $s=0$ corresponds to an ordinary 
Heaviside function. 

(More formal definition) 
Fix an orientation of the transversal plane to $X$. 
We may assume (after renumbering and reorienting if necessary) 
that $4s+2$ chambers $C_1, C_2, \dots, C_{4s+2}$ 
are arranged counterclockwise. Further assume 
$C_1=H_1^+\cap H_2^+, 
C_2=H_2^-\cap H_3^-, 
C_3=H_3^+\cap H_4^+, 
\dots, C_{2s+1}=H_{2s+1}^+\cap H_1^+, 
C_{2s+2}=H_1^-\cap H_2^-, \dots$. 
Then, define 
\begin{equation}
\begin{aligned}
\Alt^+(H_1, \dots H_{2s+1})
:=&1_{C_1}+1_{C_3}+1_{C_5}+\cdots+1_{C_{4s+1}}\\
=&
x_1^+x_2^++x_3^+x_4^++\cdots x_{2s-1}^+x_{2s}^++x_{2s+1}^+x_1^+\\
=&x_1^++x_2^++\cdots+x_{2s+1}^+-s, 
\\
\Alt^-(H_1, \dots H_{2s+1})
:=&1-\Alt^+(H_1, \dots H_{2s+1}). 
\end{aligned}
\end{equation}
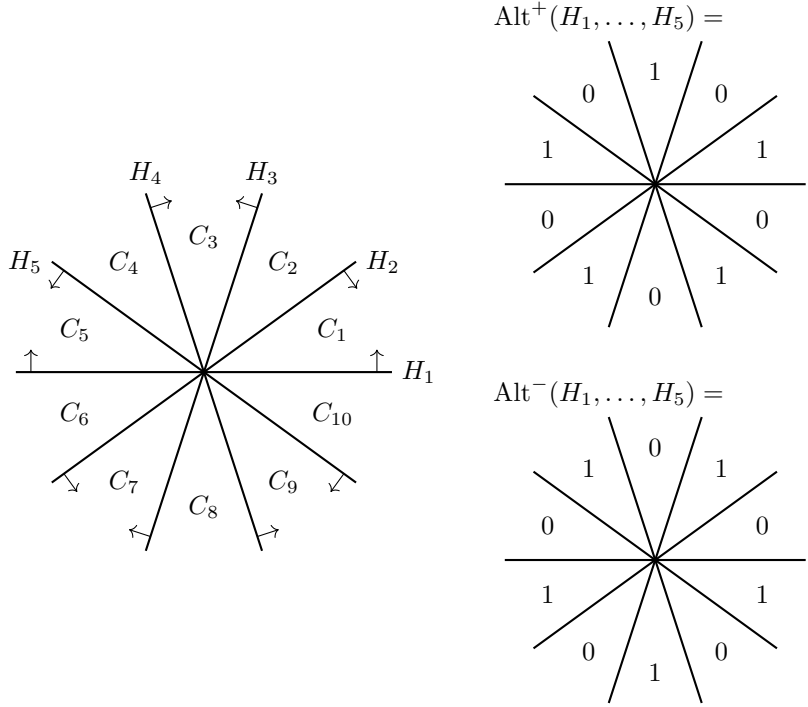
\begin{figure}[htbp]
\centering
\begin{tikzpicture}

\coordinate (A1) at (0,0); 
\coordinate (A2) at (6,2.5); 
\coordinate (A3) at (6,-2.5); 

\draw[thick] (A1)--++(0:2.5) node[right] {$H_1$};
\draw[thick] (A1)--++(36:2.5) node[right] {$H_2$};
\draw[thick] (A1)--++(72:2.5) node[above] {$H_3$};
\draw[thick] (A1)--++(108:2.5) node[above] {$H_4$};
\draw[thick] (A1)--++(144:2.5) node[left] {$H_5$};
\draw[thick] (A1)--++(180:2.5);
\draw[thick] (A1)--++(216:2.5);
\draw[thick] (A1)--++(252:2.5);
\draw[thick] (A1)--++(288:2.5);
\draw[thick] (A1)--++(324:2.5);
\draw[->] ($(A1)+(0:2.3)$) --++(90:0.3);
\draw[->] ($(A1)+(36:2.3)$) --++(306:0.3);
\draw[->] ($(A1)+(72:2.3)$) --++(162:0.3);
\draw[->] ($(A1)+(108:2.3)$) --++(18:0.3);
\draw[->] ($(A1)+(144:2.3)$) --++(234:0.3);
\draw[->] ($(A1)+(180:2.3)$) --++(90:0.3);
\draw[->] ($(A1)+(216:2.3)$) --++(306:0.3);
\draw[->] ($(A1)+(252:2.3)$) --++(162:0.3);
\draw[->] ($(A1)+(288:2.3)$) --++(18:0.3);
\draw[->] ($(A1)+(324:2.3)$) --++(234:0.3);
\draw ($(A1)+(18:1.8)$) node {$C_1$};
\draw ($(A1)+(54:1.8)$) node {$C_2$};
\draw ($(A1)+(90:1.8)$) node {$C_3$};
\draw ($(A1)+(126:1.8)$) node {$C_4$};
\draw ($(A1)+(162:1.8)$) node {$C_5$};
\draw ($(A1)+(198:1.8)$) node {$C_6$};
\draw ($(A1)+(234:1.8)$) node {$C_7$};
\draw ($(A1)+(270:1.8)$) node {$C_8$};
\draw ($(A1)+(306:1.8)$) node {$C_9$};
\draw ($(A1)+(342:1.8)$) node {$C_{10}$};

\draw[thick] (A2)--++(0:2);
\draw[thick] (A2)--++(36:2);
\draw[thick] (A2)--++(72:2);
\draw[thick] (A2)--++(108:2) node[above] {$\Alt^+(H_1, \dots, H_5)=$};
\draw[thick] (A2)--++(144:2);
\draw[thick] (A2)--++(180:2);
\draw[thick] (A2)--++(216:2);
\draw[thick] (A2)--++(252:2);
\draw[thick] (A2)--++(288:2);
\draw[thick] (A2)--++(324:2);
\draw ($(A2)+(18:1.5)$) node {$1$};
\draw ($(A2)+(54:1.5)$) node {$0$};
\draw ($(A2)+(90:1.5)$) node {$1$};
\draw ($(A2)+(126:1.5)$) node {$0$};
\draw ($(A2)+(162:1.5)$) node {$1$};
\draw ($(A2)+(198:1.5)$) node {$0$};
\draw ($(A2)+(234:1.5)$) node {$1$};
\draw ($(A2)+(270:1.5)$) node {$0$};
\draw ($(A2)+(306:1.5)$) node {$1$};
\draw ($(A2)+(342:1.5)$) node {$0$};

\draw[thick] (A3)--++(0:2);
\draw[thick] (A3)--++(36:2);
\draw[thick] (A3)--++(72:2);
\draw[thick] (A3)--++(108:2) node[above] {$\Alt^-(H_1, \dots, H_5)=$};
\draw[thick] (A3)--++(144:2);
\draw[thick] (A3)--++(180:2);
\draw[thick] (A3)--++(216:2);
\draw[thick] (A3)--++(252:2);
\draw[thick] (A3)--++(288:2);
\draw[thick] (A3)--++(324:2);
\draw ($(A3)+(18:1.5)$) node {$0$};
\draw ($(A3)+(54:1.5)$) node {$1$};
\draw ($(A3)+(90:1.5)$) node {$0$};
\draw ($(A3)+(126:1.5)$) node {$1$};
\draw ($(A3)+(162:1.5)$) node {$0$};
\draw ($(A3)+(198:1.5)$) node {$1$};
\draw ($(A3)+(234:1.5)$) node {$0$};
\draw ($(A3)+(270:1.5)$) node {$1$};
\draw ($(A3)+(306:1.5)$) node {$0$};
\draw ($(A3)+(342:1.5)$) node {$1$};

\end{tikzpicture}
\caption{$\Alt^+(H_1, \dots, H_{2s+1})$ and 
$\Alt^-(H_1, \dots, H_{2s+1})$ ($s=2$).}
\label{fig:Alt}
\end{figure}
Note that this construction works only when the number 
of hyperplanes $H_1, \dots, H_{2s+1}$ is odd. 
\end{example}

Before stating the main result, recall Sylvester-Gallai 
Theorem (see, e.g., \cite[Proposition 6.1.1]{ori-mat}): Let $L_1, \dots, L_k\subset \mathbb{RP}^2$ be 
lines in the real projective plane $\mathbb{RP}^2$. 
Then it has either the intersection 
$\bigcap_{i=1}^kL_i\neq\emptyset$ (pencil type), 
or there exists a point $p\in\mathbb{RP}^2$ through which 
exactly two of the lines pass. 
From the Sylvester-Gallai theorem, we have the following. 
\begin{lemma}
{\label{lem:SG}}
Let $\A=\{H_1, \dots, H_n\}$ be an arrangement in $\R^\ell$ 
with $\ell\geq 3$. Then either one of the following 
$(a)$ or $(b)$ holds. 
\begin{itemize}
    \item[$(a)$] The intersection $\cap\A=\bigcap_{H\in\A}H$ has codimension at most $2$.  
    \item[$(b)$] There exists a pair of hyperplanes 
    $H, H'\in\A$ such that no other hyperplane $H''\in\A$ contains $H\cap H'$. 
\end{itemize}
\end{lemma}
\begin{proof}
Suppose that $\cap\A$ has codimension at least $3$. 
Choose $X\in L(\A)$ with $\codim X=3$. 
Let $p\in X$ be a generic point and let $D\simeq 
D^3$ be a $3$-dimensional disk transversal to $X$ with 
$D\cap X=\{p\}$. 
Then the arrangement $\A\cap D$ is isomorphic to 
the localization $\A_X$. 
By the Sylvester-Gallai Theorem, 
there exist $H, H'\in\A_X$ 
such that no other hyperplanes in $\A$ contain $H\cap H'$. 
\end{proof}

The following theorem shows that every element in $\gHeav(\A)$ 
is of the form $\Alt^\pm(H_1, \dots, H_{2s+1})$. 

\begin{theorem}
\label{thm:genHeav}
Assume $\cha R\neq 2$. 

(1) Let $y\in\Fil^1\VG(\A)$ with 
$\Supp (y)=\{H_{i_1}, \dots, H_{i_r}\}$.  
If $y$ is not a constant and $y^2\in\Fil^1\VG(\A)$, 
then $r$ is odd with the intersection 
$\cap\Supp(y):=\bigcap_{H\in\Supp(y)}H$ has codimension 
at most $2$ and $y$ is expressed as $y=c_1\Alt^+(H_{i_1}, \dots, H_{i_r})+c_2$ for some $c_1, c_2\in R$.

(2) 
Let $y\in\gHeav(\A)$. Then there exist an odd number of 
hyperplanes $H_1, \dots, H_{2s+1}\in\A$ ($s\geq 0$) such that 
the intersection 
$X=\bigcap_{i=1}^{2s+1}H_i$ has codimension at most $2$, and 
$y=\Alt^\pm(H_1, \dots, H_{2s+1})$. 
\end{theorem}

\begin{proof}
(1) Let $y\in\Fil^1$. Assume $y^2\in\Fil^1$ and $y$ is not constant. 
Let $H\in\Supp(y)$. 
By Proposition \ref{prop:vgrho}, $\rho_H(y)=\gamma_1$ 
is a nonzero constant function on $H$. 
Write $\ch(\A^H)=\{D_1, \dots, D_s\}$, and set 
$\alpha_i=y(D_i^+)$, and 
$\beta_i=y(D_i^-)$ (see Figure \ref{fig:sqzero}). 
\begin{figure}[htbp]
\centering
\begin{tikzpicture}

\draw[ultra thick] (-1,1) -- (8,1) node[right] {$H$};
\draw[-stealth, thick] (7.8, 1) --++(0, 0.3);
\draw (0.5,0) -- ++(1.75,3.5);
\draw (1,0) -- ++(0,3.5);
\draw (5,0) -- ++(-3.5,3.5);
\draw (6.5,0) -- ++(-1.75,3.5);

\draw (-0.5,1) node[below] {\footnotesize $D_1$}; 
\draw (2,1) node[below] {\footnotesize $D_2$}; 
\draw (5,1) node[below] {\footnotesize $D_3$}; 
\draw (7,1) node[below] {\footnotesize $D_4$}; 

\draw (0,2) node {$\alpha_1$}; 
\draw (-0.5,0) node {$\beta_1$}; 
\draw (2.3,2) node {$\alpha_2$}; 
\draw (3,0) node {$\beta_2$}; 
\draw (4.5,2) node {$\alpha_3$}; 
\draw (5.5,0) node {$\beta_3$}; 
\draw (6.5,2) node {$\alpha_4$}; 
\draw (7.5,0) node {$\beta_4$}; 

\filldraw[draw=black, fill=white] (1,1) circle [radius=0.1]; 
\filldraw[draw=black, fill=white] (4,1) circle [radius=0.1]; 
\filldraw[draw=black, fill=white] (6,1) circle [radius=0.1]; 

\end{tikzpicture}
\caption{Square-zero elements.}
\label{fig:sqzero}
\end{figure}
Then, 
\begin{equation}
\label{eq:constdiff}
\alpha_1-\beta_1=\alpha_2-\beta_2=\cdots
=\alpha_s-\beta_s=\gamma_1
\end{equation}
is a nonzero constant. Since $y^2\in\Fil^1$, 
again by Proposition \ref{prop:vgrho}, there is a 
constant $\gamma_2$ satisfying 
\begin{equation}
\label{eq:squarediff}
\alpha_1^2-\beta_1^2=\alpha_2^2-\beta_2^2=\cdots
=\alpha_s^2-\beta_s^2=\gamma_2. 
\end{equation}
We also have 
\begin{equation}
\label{eq:constsum}
\alpha_1+\beta_1=\alpha_2+\beta_2=\cdots
=\alpha_s+\beta_s=\gamma_3, 
\end{equation}
where $\gamma_3=\gamma_2/\gamma_1$. 
Since $\cha R\neq 2$, we have 
\begin{equation}
\label{eq:sol}
\alpha_1=\alpha_2=\cdots=\alpha_s=\frac{\gamma_1+\gamma_3}{2}, \ \ 
\beta_1=\beta_2=\cdots=\beta_s=\frac{\gamma_3-\gamma_1}{2}. 
\end{equation}
By Lemma \ref{lem:SG}, either 
\begin{itemize}
\item[$(a)$] 
$\codim\bigcap_{H\in\Supp(y)}H\leq 2$, or 
\item[$(b)$]  
there exist $H, H'\in\Supp(y)$ such that no other hyperplane 
in $\Supp(y)$ contains $H\cap H'$, 
\end{itemize}
holds. Case $(b)$ contradicts the fact that $H, H'\in\Supp(y)$ 
(See Figure \ref{fig:aroundp}, by the above arguments, 
$y(C_1)=y(C_2)=y(C_3)=y(C_4)$ holds). Thus $(a)$ holds. 
It is easily seen that $\#\Supp(y)$ must be odd; otherwise 
one can not assign the two constant values consistently. 
Hence $\#\Supp(y)=2s+1$ and the values alternate 
around the codimension $2$ intersection (Figure \ref{fig:sqzeroAlt}). 
Therefore, $y$ can be expressed as 
$y=c_1\Alt^+(H_{i_1}, \dots, H_{i_r})+c_2$ for some 
$c_1, c_2\in R$.
\begin{figure}[htbp]
\centering
\begin{tikzpicture}

\coordinate (A1) at (0,0); 
\coordinate (A2) at (6,2.5); 
\coordinate (A3) at (6,-2.5); 

\draw[thick] (A1)--++(0:2.5) node[right] {$H_1$};
\draw[thick] (A1)--++(36:2.5) node[right] {$H_2$};
\draw[thick] (A1)--++(72:2.5) node[above] {$H_3$};
\draw[thick] (A1)--++(108:2.5) node[above] {$H_4$};
\draw[thick] (A1)--++(144:2.5) node[left] {$H_5$};
\draw[thick] (A1)--++(180:2.5);
\draw[thick] (A1)--++(216:2.5);
\draw[thick] (A1)--++(252:2.5);
\draw[thick] (A1)--++(288:2.5);
\draw[thick] (A1)--++(324:2.5);
\draw[->] ($(A1)+(0:2.3)$) --++(90:0.3);
\draw[->] ($(A1)+(36:2.3)$) --++(306:0.3);
\draw[->] ($(A1)+(72:2.3)$) --++(162:0.3);
\draw[->] ($(A1)+(108:2.3)$) --++(18:0.3);
\draw[->] ($(A1)+(144:2.3)$) --++(234:0.3);
\draw[->] ($(A1)+(180:2.3)$) --++(90:0.3);
\draw[->] ($(A1)+(216:2.3)$) --++(306:0.3);
\draw[->] ($(A1)+(252:2.3)$) --++(162:0.3);
\draw[->] ($(A1)+(288:2.3)$) --++(18:0.3);
\draw[->] ($(A1)+(324:2.3)$) --++(234:0.3);
\draw ($(A1)+(18:1.8)$) node {$\alpha$};
\draw ($(A1)+(54:1.8)$) node {$\beta$};
\draw ($(A1)+(90:1.8)$) node {$\alpha$};
\draw ($(A1)+(126:1.8)$) node {$\beta$};
\draw ($(A1)+(162:1.8)$) node {$\alpha$};
\draw ($(A1)+(198:1.8)$) node {$\beta$};
\draw ($(A1)+(234:1.8)$) node {$\alpha$};
\draw ($(A1)+(270:1.8)$) node {$\beta$};
\draw ($(A1)+(306:1.8)$) node {$\alpha$};
\draw ($(A1)+(342:1.8)$) node {$\beta$};

\end{tikzpicture}
\caption{$y=(\alpha-\beta)\cdot\Alt^+(H_1, \dots, H_{2s+1})+\beta$.}
\label{fig:sqzeroAlt}
\end{figure}
(2) is immediate from (1). 
\end{proof}

\subsection{Conjectural algorithm}
\label{subsec:reconst}

In this section, let $R$ be an integral domain with $\cha R\neq 2$. 
Even when the arrangement $\A$ is not generic in codimension $2$, 
the filtered algebra $\VG(\A)$ may still determine the tope graph 
$\Tope(\A)$. 
We formulate the following conjecture. 

\begin{conjecture}
\label{conj:reconst}
Assume $\cha R\neq 2$. 
Let $\A_1, \A_2$ be real arrangements. If $\VG(\A_1)$ and 
$\VG(\A_2)$ are isomorphic as filtered algebras 
(or, if $\grVG^\bullet(\A_1)$ and 
$\grVG^\bullet(\A_2)$ are isomorphic as graded algebras), then 
tope graphs $\Tope(\A_1)$ and $\Tope(\A_2)$ are isomorphic. 
\end{conjecture}
By Theorem \ref{thm:gencodim2}, Conjecture \ref{conj:reconst} holds 
for arrangements that are generic in codimension $2$ when 
$\cha R\neq 2$. As seen in \S \ref{subsec:gheav}, however, 
there are many generalized Heaviside 
functions that are not Heaviside functions. 
To formulate a reconstruction procedure from $\VG(\A)$, 
we introduce the notion of ``generalized tope graphs.'' 

\begin{definition}
Let $y_1, \dots, y_p\in\gHeav(\A)$. 
The \emph{generalized tope graph} 
$\gTope(y_1, \dots, y_p)$ is the graph defined as follows: 
\begin{itemize}
\item 
The vertex set is the set of primitive idempotents 
$\PrimIdem(\VG(\A))$. 
\item 
A pair $\{v, v'\}\subset\PrimIdem(\VG(\A))$ forms an edge 
if and only if 
\begin{equation}
\label{eq:sep}
\#\{i\in [p]\mid y_iv\neq y_iv'\}=1. 
\end{equation}
\end{itemize}
\end{definition}

\begin{conjecture}[Non-deterministic reconstruction]
\label{conj:algorithm}
Let $\A=\{H_1, \dots, H_n\}$ be a real arrangement in $\R^\ell$. 
Perform the following steps:  
\begin{itemize}
\item[(Step 1)] 
Choose elements $y_1, \dots, y_n\in\gHeav(\A)$ such that 
$1, y_1, \dots, y_n$ form a basis of $\Fil^1\VG(\A)$. 
\item[(Step 2)] 
If the generalized tope graph $\gTope(y_1, \dots, y_n)$ is the 
tope graph of some oriented matroid, stop. 
Otherwise, return to (Step 1) with a different choice of 
$y_1, \dots, y_n$. 
\end{itemize}
Then the resulting graph $\gTope(y_1, \dots, y_n)$ 
is isomorphic to the tope graph $\Tope(\A)$. 
\end{conjecture}

Since $n=\#\A$ is determined purely from 
the filtered algebra structure 
(Proposition \ref{prop:recrank}) and $\gHeav(\A)$ is a finite set, 
the procedure above terminates in finitely many trials.  
Moreover, if the chosen $y_1,\dots,y_n$ happen to be the Heaviside
functions, then the graph obtained in (Step 2) is exactly the tope 
graph of $\A$, because condition~\eqref{eq:sep} coincides with
adjacency of chambers (Proposition \ref{prop:rectopeg}).

It is worth noting that even if $y_1, \dots, y_n$ are not (all) 
Heaviside functions, the graph $\gTope(y_1,\dots,y_n)$ may 
still be a tope graph of an oriented matroid 
(see the example below), and in many
examples it is isomorphic to the original $\Tope(\A)$. 

\begin{example}
\label{ex:nonHeav}
Let $\A=\{H_1, H_2, \dots, H_6\}$ be the 
$A_3$-arrangement in $\R^3$ 
defined by $xyz(x-y)(x-z)(y-z)=0$. The characteristic polynomial is 
$\chi(\A, t)=(t-1)(t-2)(t-3)$, and $\#\ch(\A)=24$. 
We give numberings of chambers $1, 2, \dots, 12, 1', 2', \dots, 12'$ 
as in Figure \ref{fig:A3chambers}. 
\begin{figure}[htbp]
\centering
\begin{tikzpicture}[scale=0.95]

\coordinate (P1) at (2,2); 
\coordinate (P2) at (9,2); 

\draw [very thick] ($(P1)+(210:3.5)$) -- ($(P1)+(30:2.5)$) node[right] {$H_2$} ;
\draw [very thick] ($(P1)+(270:2.5)$) -- ($(P1)+(90:3.5)$) node[above] {$H_4$} ;
\draw [very thick] ($(P1)+(330:3.5)$) -- ($(P1)+(150:2.5)$) node[left] {$H_6$} ;
\draw [very thick] ($(P1)+(30:1)+(120:3)$) node[left] {$H_5$}  -- ($(P1)+(30:1)+(300:3)$);
\draw [very thick] ($(P1)+(150:1)+(240:3)$) -- ($(P1)+(150:1)+(60:3)$) node[right] {$H_3$} ;
\draw [very thick] ($(P1)+(270:1)+(0:3)$) node[right] {$H_1$} -- ($(P1)+(270:1)+(180:3)$);

\draw [->, thick] ($(P1)+(270:1)+(0:2.8)$)  -- ++(270:0.45);
\draw [->, thick] ($(P1)+(30:2.3)$) -- ++(300:0.5);
\draw [->, thick] ($(P1)+(150:1)+(60:2.7)$) -- ++(330:0.5);
\draw [->, thick] ($(P1)+(90:3.2)$) -- ++(0:0.5) ;
\draw [->, thick] ($(P1)+(30:1)+(120:2.9)$) -- ++(30:0.5);
\draw [->, thick] ($(P1)+(150:2.2)$) -- ++ (60:0.5);

\draw [very thick] ($(P2)+(210:3.5)$) -- ($(P2)+(30:2.5)$)  ;
\draw [very thick] ($(P2)+(270:2.5)$) -- ($(P2)+(90:3.5)$)  ;
\draw [very thick] ($(P2)+(330:3.5)$) -- ($(P2)+(150:2.5)$)  ;
\draw [very thick] ($(P2)+(30:1)+(120:3)$)  -- ($(P2)+(30:1)+(300:3)$);
\draw [very thick] ($(P2)+(150:1)+(240:3)$) -- ($(P2)+(150:1)+(60:3)$);
\draw [very thick] ($(P2)+(270:1)+(0:3)$)  -- ($(P2)+(270:1)+(180:3)$);
\draw ($(P2)+(2,0)$) node {$1$};
\draw ($(P2)+(1,2)$) node {$2$};
\draw ($(P2)+(0.3,3)$) node {$3$};
\draw ($(P2)+(-0.3,3)$) node {$4$};
\draw ($(P2)+(-1,2)$) node {$5$};
\draw ($(P2)+(-2,0)$) node {$6$};

\draw ($(P2)+(-3,-1.35)$) node {$1'$};
\draw ($(P2)+(-2.7,-2)$) node {$2'$};
\draw ($(P2)+(-1,-2)$) node {$3'$};
\draw ($(P2)+(1,-2)$) node {$4'$};
\draw ($(P2)+(2.7,-2)$) node {$5'$};
\draw ($(P2)+(3,-1.35)$) node {$6'$};

\draw ($(P2)+(0.8,0)$) node {$7$};
\draw ($(P2)+(0.4,0.6)$) node {$8$};
\draw ($(P2)+(0.4,-0.6)$) node {$12$};
\draw ($(P2)+(-0.8,0)$) node {$10$};
\draw ($(P2)+(-0.4,0.6)$) node {$9$};
\draw ($(P2)+(-0.4,-0.6)$) node {$11$};

\end{tikzpicture}
\caption{$\A=A_3$ and numbering of chambers (the chamber $i'$ is the opposite chamber of the chamber $i$). The remaining chambers are 
$7', 8', 9', 10', 11', 12'$.}
\label{fig:A3chambers}
\end{figure}
Let $x_1, \dots, x_6$ be the Heaviside functions and define 
generalized Heaviside functions $y_1, y_2, y_3, y_4$ as 
in Figure \ref{fig:A3gHeav}. 
\begin{figure}[htbp]
\centering
\begin{tikzpicture}[scale=0.95]

\coordinate (P2) at (9,4); 

\coordinate(Q1) at (0,0);
\coordinate(Q2) at (4,0);
\coordinate(Q3) at (8,0);
\coordinate(Q4) at (0,-4);
\coordinate(Q5) at (4,-4);
\coordinate(Q6) at (8,-4);
\coordinate(Q7) at (0,-8);
\coordinate(Q8) at (4,-8);
\coordinate(Q9) at (8,-8);
\coordinate(Q10) at (0,-12);

\draw [very thin] ($(Q1)+(210:1.75)$) -- ($(Q1)+(30:1.25)$)  ;
\draw [very thin] ($(Q1)+(270:1.2)$) node[below] {$x_1$} -- ($(Q1)+(90:1.75)$)  ;
\draw [very thin] ($(Q1)+(330:1.75)$) -- ($(Q1)+(150:1.25)$)  ;
\draw [very thin] ($(Q1)+(30:0.5)+(120:1.5)$)  -- ($(Q1)+(30:0.5)+(300:1.5)$);
\draw [very thin] ($(Q1)+(150:0.5)+(240:1.5)$) -- ($(Q1)+(150:0.5)+(60:1.5)$);
\draw [very thick] ($(Q1)+(270:0.5)+(0:1.5)$)  -- ($(Q1)+(270:0.5)+(180:1.5)$);
\draw ($(Q1)+(1,0)$) node {\footnotesize $0$};
\draw ($(Q1)+(0.5,1)$) node {\footnotesize $0$};
\draw ($(Q1)+(0.15,1.5)$) node {\footnotesize $0$};
\draw ($(Q1)+(-0.15,1.5)$) node {\footnotesize $0$};
\draw ($(Q1)+(-0.5,1)$) node {\footnotesize $0$};
\draw ($(Q1)+(-1,0)$) node {\footnotesize $0$};

\draw ($(Q1)+(-1.5,-0.675)$) node {\footnotesize $1$};
\draw ($(Q1)+(-1.35,-1)$) node {\footnotesize $1$};
\draw ($(Q1)+(-0.5,-1)$) node {\footnotesize $1$};
\draw ($(Q1)+(0.5,-1)$) node {\footnotesize $1$};
\draw ($(Q1)+(1.35,-1)$) node {\footnotesize $1$};
\draw ($(Q1)+(1.5,-0.675)$) node {\footnotesize $1$};

\draw ($(Q1)+(0.4,0)$) node {\footnotesize $0$};
\draw ($(Q1)+(0.2,0.3)$) node {\footnotesize $0$};
\draw ($(Q1)+(0.2,-0.3)$) node {\footnotesize $0$};
\draw ($(Q1)+(-0.4,0)$) node {\footnotesize $0$};
\draw ($(Q1)+(-0.2,0.3)$) node {\footnotesize $0$};
\draw ($(Q1)+(-0.2,-0.3)$) node {\footnotesize $0$};

\draw [very thick] ($(Q2)+(210:1.75)$) -- ($(Q2)+(30:1.25)$)  ;
\draw [very thin] ($(Q2)+(270:1.2)$) node[below] {$x_2$} -- ($(Q2)+(90:1.75)$)  ;
\draw [very thin] ($(Q2)+(330:1.75)$) -- ($(Q2)+(150:1.25)$)  ;
\draw [very thin] ($(Q2)+(30:0.5)+(120:1.5)$)  -- ($(Q2)+(30:0.5)+(300:1.5)$);
\draw [very thin] ($(Q2)+(150:0.5)+(240:1.5)$) -- ($(Q2)+(150:0.5)+(60:1.5)$);
\draw [very thin] ($(Q2)+(270:0.5)+(0:1.5)$)  -- ($(Q2)+(270:0.5)+(180:1.5)$);
\draw ($(Q2)+(1,0)$) node {\footnotesize $1$};
\draw ($(Q2)+(0.5,1)$) node {\footnotesize $0$};
\draw ($(Q2)+(0.15,1.5)$) node {\footnotesize $0$};
\draw ($(Q2)+(-0.15,1.5)$) node {\footnotesize $0$};
\draw ($(Q2)+(-0.5,1)$) node {\footnotesize $0$};
\draw ($(Q2)+(-1,0)$) node {\footnotesize $0$};

\draw ($(Q2)+(-1.5,-0.675)$) node {\footnotesize $0$};
\draw ($(Q2)+(-1.35,-1)$) node {\footnotesize $1$};
\draw ($(Q2)+(-0.5,-1)$) node {\footnotesize $1$};
\draw ($(Q2)+(0.5,-1)$) node {\footnotesize $1$};
\draw ($(Q2)+(1.35,-1)$) node {\footnotesize $1$};
\draw ($(Q2)+(1.5,-0.675)$) node {\footnotesize $1$};

\draw ($(Q2)+(0.4,0)$) node {\footnotesize $1$};
\draw ($(Q2)+(0.2,0.3)$) node {\footnotesize $0$};
\draw ($(Q2)+(0.2,-0.3)$) node {\footnotesize $1$};
\draw ($(Q2)+(-0.4,0)$) node {\footnotesize $0$};
\draw ($(Q2)+(-0.2,0.3)$) node {\footnotesize $0$};
\draw ($(Q2)+(-0.2,-0.3)$) node {\footnotesize $1$};

\draw [very thin] ($(Q3)+(210:1.75)$) -- ($(Q3)+(30:1.25)$)  ;
\draw [very thin] ($(Q3)+(270:1.2)$) node[below] {$x_3$} -- ($(Q3)+(90:1.75)$)  ;
\draw [very thin] ($(Q3)+(330:1.75)$) -- ($(Q3)+(150:1.25)$)  ;
\draw [very thin] ($(Q3)+(30:0.5)+(120:1.5)$)  -- ($(Q3)+(30:0.5)+(300:1.5)$);
\draw [very thick] ($(Q3)+(150:0.5)+(240:1.5)$) -- ($(Q3)+(150:0.5)+(60:1.5)$);
\draw [very thin] ($(Q3)+(270:0.5)+(0:1.5)$)  -- ($(Q3)+(270:0.5)+(180:1.5)$);
\draw ($(Q3)+(1,0)$) node {\footnotesize $1$};
\draw ($(Q3)+(0.5,1)$) node {\footnotesize $1$};
\draw ($(Q3)+(0.15,1.5)$) node {\footnotesize $0$};
\draw ($(Q3)+(-0.15,1.5)$) node {\footnotesize $0$};
\draw ($(Q3)+(-0.5,1)$) node {\footnotesize $0$};
\draw ($(Q3)+(-1,0)$) node {\footnotesize $0$};

\draw ($(Q3)+(-1.5,-0.675)$) node {\footnotesize $0$};
\draw ($(Q3)+(-1.35,-1)$) node {\footnotesize $0$};
\draw ($(Q3)+(-0.5,-1)$) node {\footnotesize $1$};
\draw ($(Q3)+(0.5,-1)$) node {\footnotesize $1$};
\draw ($(Q3)+(1.35,-1)$) node {\footnotesize $1$};
\draw ($(Q3)+(1.5,-0.675)$) node {\footnotesize $1$};

\draw ($(Q3)+(0.4,0)$) node {\footnotesize $1$};
\draw ($(Q3)+(0.2,0.3)$) node {\footnotesize $1$};
\draw ($(Q3)+(0.2,-0.3)$) node {\footnotesize $1$};
\draw ($(Q3)+(-0.4,0)$) node {\footnotesize $1$};
\draw ($(Q3)+(-0.2,0.3)$) node {\footnotesize $1$};
\draw ($(Q3)+(-0.2,-0.3)$) node {\footnotesize $1$};

\draw [very thin] ($(Q4)+(210:1.75)$) -- ($(Q4)+(30:1.25)$)  ;
\draw [very thick] ($(Q4)+(270:1.2)$) node[below] {$x_4$} -- ($(Q4)+(90:1.75)$)  ;
\draw [very thin] ($(Q4)+(330:1.75)$) -- ($(Q4)+(150:1.25)$)  ;
\draw [very thin] ($(Q4)+(30:0.5)+(120:1.5)$)  -- ($(Q4)+(30:0.5)+(300:1.5)$);
\draw [very thin] ($(Q4)+(150:0.5)+(240:1.5)$) -- ($(Q4)+(150:0.5)+(60:1.5)$);
\draw [very thin] ($(Q4)+(270:0.5)+(0:1.5)$)  -- ($(Q4)+(270:0.5)+(180:1.5)$);
\draw ($(Q4)+(1,0)$) node {\footnotesize $1$};
\draw ($(Q4)+(0.5,1)$) node {\footnotesize $1$};
\draw ($(Q4)+(0.15,1.5)$) node {\footnotesize $1$};
\draw ($(Q4)+(-0.15,1.5)$) node {\footnotesize $0$};
\draw ($(Q4)+(-0.5,1)$) node {\footnotesize $0$};
\draw ($(Q4)+(-1,0)$) node {\footnotesize $0$};

\draw ($(Q4)+(-1.5,-0.675)$) node {\footnotesize $0$};
\draw ($(Q4)+(-1.35,-1)$) node {\footnotesize $0$};
\draw ($(Q4)+(-0.5,-1)$) node {\footnotesize $0$};
\draw ($(Q4)+(0.5,-1)$) node {\footnotesize $1$};
\draw ($(Q4)+(1.35,-1)$) node {\footnotesize $1$};
\draw ($(Q4)+(1.5,-0.675)$) node {\footnotesize $1$};

\draw ($(Q4)+(0.4,0)$) node {\footnotesize $1$};
\draw ($(Q4)+(0.2,0.3)$) node {\footnotesize $1$};
\draw ($(Q4)+(0.2,-0.3)$) node {\footnotesize $1$};
\draw ($(Q4)+(-0.4,0)$) node {\footnotesize $0$};
\draw ($(Q4)+(-0.2,0.3)$) node {\footnotesize $0$};
\draw ($(Q4)+(-0.2,-0.3)$) node {\footnotesize $0$};

\draw [very thin] ($(Q5)+(210:1.75)$) -- ($(Q5)+(30:1.25)$)  ;
\draw [very thin] ($(Q5)+(270:1.2)$) node[below] {$x_5$} -- ($(Q5)+(90:1.75)$)  ;
\draw [very thin] ($(Q5)+(330:1.75)$) -- ($(Q5)+(150:1.25)$)  ;
\draw [very thick] ($(Q5)+(30:0.5)+(120:1.5)$)  -- ($(Q5)+(30:0.5)+(300:1.5)$);
\draw [very thin] ($(Q5)+(150:0.5)+(240:1.5)$) -- ($(Q5)+(150:0.5)+(60:1.5)$);
\draw [very thin] ($(Q5)+(270:0.5)+(0:1.5)$)  -- ($(Q5)+(270:0.5)+(180:1.5)$);
\draw ($(Q5)+(1,0)$) node {\footnotesize $1$};
\draw ($(Q5)+(0.5,1)$) node {\footnotesize $1$};
\draw ($(Q5)+(0.15,1.5)$) node {\footnotesize $1$};
\draw ($(Q5)+(-0.15,1.5)$) node {\footnotesize $1$};
\draw ($(Q5)+(-0.5,1)$) node {\footnotesize $0$};
\draw ($(Q5)+(-1,0)$) node {\footnotesize $0$};

\draw ($(Q5)+(-1.5,-0.675)$) node {\footnotesize $0$};
\draw ($(Q5)+(-1.35,-1)$) node {\footnotesize $0$};
\draw ($(Q5)+(-0.5,-1)$) node {\footnotesize $0$};
\draw ($(Q5)+(0.5,-1)$) node {\footnotesize $0$};
\draw ($(Q5)+(1.35,-1)$) node {\footnotesize $1$};
\draw ($(Q5)+(1.5,-0.675)$) node {\footnotesize $1$};

\draw ($(Q5)+(0.4,0)$) node {\footnotesize $0$};
\draw ($(Q5)+(0.2,0.3)$) node {\footnotesize $0$};
\draw ($(Q5)+(0.2,-0.3)$) node {\footnotesize $0$};
\draw ($(Q5)+(-0.4,0)$) node {\footnotesize $0$};
\draw ($(Q5)+(-0.2,0.3)$) node {\footnotesize $0$};
\draw ($(Q5)+(-0.2,-0.3)$) node {\footnotesize $0$};

\draw [very thin] ($(Q6)+(210:1.75)$) -- ($(Q6)+(30:1.25)$)  ;
\draw [very thin] ($(Q6)+(270:1.2)$) node[below] {$x_6$} -- ($(Q6)+(90:1.75)$)  ;
\draw [very thick] ($(Q6)+(330:1.75)$) -- ($(Q6)+(150:1.25)$)  ;
\draw [very thin] ($(Q6)+(30:0.5)+(120:1.5)$)  -- ($(Q6)+(30:0.5)+(300:1.5)$);
\draw [very thin] ($(Q6)+(150:0.5)+(240:1.5)$) -- ($(Q6)+(150:0.5)+(60:1.5)$);
\draw [very thin] ($(Q6)+(270:0.5)+(0:1.5)$)  -- ($(Q6)+(270:0.5)+(180:1.5)$);
\draw ($(Q6)+(1,0)$) node {\footnotesize $1$};
\draw ($(Q6)+(0.5,1)$) node {\footnotesize $1$};
\draw ($(Q6)+(0.15,1.5)$) node {\footnotesize $1$};
\draw ($(Q6)+(-0.15,1.5)$) node {\footnotesize $1$};
\draw ($(Q6)+(-0.5,1)$) node {\footnotesize $1$};
\draw ($(Q6)+(-1,0)$) node {\footnotesize $0$};

\draw ($(Q6)+(-1.5,-0.675)$) node {\footnotesize $0$};
\draw ($(Q6)+(-1.35,-1)$) node {\footnotesize $0$};
\draw ($(Q6)+(-0.5,-1)$) node {\footnotesize $0$};
\draw ($(Q6)+(0.5,-1)$) node {\footnotesize $0$};
\draw ($(Q6)+(1.35,-1)$) node {\footnotesize $0$};
\draw ($(Q6)+(1.5,-0.675)$) node {\footnotesize $1$};

\draw ($(Q6)+(0.4,0)$) node {\footnotesize $1$};
\draw ($(Q6)+(0.2,0.3)$) node {\footnotesize $1$};
\draw ($(Q6)+(0.2,-0.3)$) node {\footnotesize $0$};
\draw ($(Q6)+(-0.4,0)$) node {\footnotesize $0$};
\draw ($(Q6)+(-0.2,0.3)$) node {\footnotesize $1$};
\draw ($(Q6)+(-0.2,-0.3)$) node {\footnotesize $0$};

\draw [very thick] ($(Q7)+(210:1.75)$) -- ($(Q7)+(30:1.25)$)  ;
\draw [very thin] ($(Q7)+(270:1.2)$) node[below] {$y_1$} -- ($(Q7)+(90:1.75)$)  ;
\draw [very thin] ($(Q7)+(330:1.75)$) -- ($(Q7)+(150:1.25)$)  ;
\draw [very thin] ($(Q7)+(30:0.5)+(120:1.5)$)  -- ($(Q7)+(30:0.5)+(300:1.5)$);
\draw [very thick] ($(Q7)+(150:0.5)+(240:1.5)$) -- ($(Q7)+(150:0.5)+(60:1.5)$);
\draw [very thick] ($(Q7)+(270:0.5)+(0:1.5)$)  -- ($(Q7)+(270:0.5)+(180:1.5)$);
\draw ($(Q7)+(1,0)$) node {\footnotesize $1$};
\draw ($(Q7)+(0.5,1)$) node {\footnotesize $0$};
\draw ($(Q7)+(0.15,1.5)$) node {\footnotesize $1$};
\draw ($(Q7)+(-0.15,1.5)$) node {\footnotesize $1$};
\draw ($(Q7)+(-0.5,1)$) node {\footnotesize $1$};
\draw ($(Q7)+(-1,0)$) node {\footnotesize $1$};

\draw ($(Q7)+(-1.5,-0.675)$) node {\footnotesize $0$};
\draw ($(Q7)+(-1.35,-1)$) node {\footnotesize $1$};
\draw ($(Q7)+(-0.5,-1)$) node {\footnotesize $0$};
\draw ($(Q7)+(0.5,-1)$) node {\footnotesize $0$};
\draw ($(Q7)+(1.35,-1)$) node {\footnotesize $0$};
\draw ($(Q7)+(1.5,-0.675)$) node {\footnotesize $0$};

\draw ($(Q7)+(0.4,0)$) node {\footnotesize $1$};
\draw ($(Q7)+(0.2,0.3)$) node {\footnotesize $0$};
\draw ($(Q7)+(0.2,-0.3)$) node {\footnotesize $1$};
\draw ($(Q7)+(-0.4,0)$) node {\footnotesize $0$};
\draw ($(Q7)+(-0.2,0.3)$) node {\footnotesize $0$};
\draw ($(Q7)+(-0.2,-0.3)$) node {\footnotesize $1$};

\draw [very thin] ($(Q8)+(210:1.75)$) -- ($(Q8)+(30:1.25)$)  ;
\draw [very thick] ($(Q8)+(270:1.2)$) node[below] {$y_2$} -- ($(Q8)+(90:1.75)$)  ;
\draw [very thin] ($(Q8)+(330:1.75)$) -- ($(Q8)+(150:1.25)$)  ;
\draw [very thick] ($(Q8)+(30:0.5)+(120:1.5)$)  -- ($(Q8)+(30:0.5)+(300:1.5)$);
\draw [very thick] ($(Q8)+(150:0.5)+(240:1.5)$) -- ($(Q8)+(150:0.5)+(60:1.5)$);
\draw [very thin] ($(Q8)+(270:0.5)+(0:1.5)$)  -- ($(Q8)+(270:0.5)+(180:1.5)$);
\draw ($(Q8)+(1,0)$) node {\footnotesize $1$};
\draw ($(Q8)+(0.5,1)$) node {\footnotesize $1$};
\draw ($(Q8)+(0.15,1.5)$) node {\footnotesize $0$};
\draw ($(Q8)+(-0.15,1.5)$) node {\footnotesize $1$};
\draw ($(Q8)+(-0.5,1)$) node {\footnotesize $0$};
\draw ($(Q8)+(-1,0)$) node {\footnotesize $0$};

\draw ($(Q8)+(-1.5,-0.675)$) node {\footnotesize $0$};
\draw ($(Q8)+(-1.35,-1)$) node {\footnotesize $0$};
\draw ($(Q8)+(-0.5,-1)$) node {\footnotesize $1$};
\draw ($(Q8)+(0.5,-1)$) node {\footnotesize $0$};
\draw ($(Q8)+(1.35,-1)$) node {\footnotesize $1$};
\draw ($(Q8)+(1.5,-0.675)$) node {\footnotesize $1$};

\draw ($(Q8)+(0.4,0)$) node {\footnotesize $0$};
\draw ($(Q8)+(0.2,0.3)$) node {\footnotesize $0$};
\draw ($(Q8)+(0.2,-0.3)$) node {\footnotesize $0$};
\draw ($(Q8)+(-0.4,0)$) node {\footnotesize $1$};
\draw ($(Q8)+(-0.2,0.3)$) node {\footnotesize $1$};
\draw ($(Q8)+(-0.2,-0.3)$) node {\footnotesize $1$};

\draw [very thin] ($(Q9)+(210:1.75)$) -- ($(Q9)+(30:1.25)$)  ;
\draw [very thin] ($(Q9)+(270:1.2)$) node[below] {$y_3$} -- ($(Q9)+(90:1.75)$)  ;
\draw [very thick] ($(Q9)+(330:1.75)$) -- ($(Q9)+(150:1.25)$)  ;
\draw [very thick] ($(Q9)+(30:0.5)+(120:1.5)$)  -- ($(Q9)+(30:0.5)+(300:1.5)$);
\draw [very thin] ($(Q9)+(150:0.5)+(240:1.5)$) -- ($(Q9)+(150:0.5)+(60:1.5)$);
\draw [very thick] ($(Q9)+(270:0.5)+(0:1.5)$)  -- ($(Q9)+(270:0.5)+(180:1.5)$);
\draw ($(Q9)+(1,0)$) node {\footnotesize $0$};
\draw ($(Q9)+(0.5,1)$) node {\footnotesize $0$};
\draw ($(Q9)+(0.15,1.5)$) node {\footnotesize $0$};
\draw ($(Q9)+(-0.15,1.5)$) node {\footnotesize $0$};
\draw ($(Q9)+(-0.5,1)$) node {\footnotesize $1$};
\draw ($(Q9)+(-1,0)$) node {\footnotesize $0$};

\draw ($(Q9)+(-1.5,-0.675)$) node {\footnotesize $1$};
\draw ($(Q9)+(-1.35,-1)$) node {\footnotesize $1$};
\draw ($(Q9)+(-0.5,-1)$) node {\footnotesize $1$};
\draw ($(Q9)+(0.5,-1)$) node {\footnotesize $1$};
\draw ($(Q9)+(1.35,-1)$) node {\footnotesize $0$};
\draw ($(Q9)+(1.5,-0.675)$) node {\footnotesize $1$};

\draw ($(Q9)+(0.4,0)$) node {\footnotesize $1$};
\draw ($(Q9)+(0.2,0.3)$) node {\footnotesize $1$};
\draw ($(Q9)+(0.2,-0.3)$) node {\footnotesize $0$};
\draw ($(Q9)+(-0.4,0)$) node {\footnotesize $0$};
\draw ($(Q9)+(-0.2,0.3)$) node {\footnotesize $1$};
\draw ($(Q9)+(-0.2,-0.3)$) node {\footnotesize $0$};

\draw [very thick] ($(Q10)+(210:1.75)$) -- ($(Q10)+(30:1.25)$)  ;
\draw [very thick] ($(Q10)+(270:1.2)$) node[below] {$y_4$} -- ($(Q10)+(90:1.75)$)  ;
\draw [very thick] ($(Q10)+(330:1.75)$) -- ($(Q10)+(150:1.25)$)  ;
\draw [very thin] ($(Q10)+(30:0.5)+(120:1.5)$)  -- ($(Q10)+(30:0.5)+(300:1.5)$);
\draw [very thin] ($(Q10)+(150:0.5)+(240:1.5)$) -- ($(Q10)+(150:0.5)+(60:1.5)$);
\draw [very thin] ($(Q10)+(270:0.5)+(0:1.5)$)  -- ($(Q10)+(270:0.5)+(180:1.5)$);
\draw ($(Q10)+(1,0)$) node {\footnotesize $0$};
\draw ($(Q10)+(0.5,1)$) node {\footnotesize $1$};
\draw ($(Q10)+(0.15,1.5)$) node {\footnotesize $1$};
\draw ($(Q10)+(-0.15,1.5)$) node {\footnotesize $0$};
\draw ($(Q10)+(-0.5,1)$) node {\footnotesize $0$};
\draw ($(Q10)+(-1,0)$) node {\footnotesize $1$};

\draw ($(Q10)+(-1.5,-0.675)$) node {\footnotesize $1$};
\draw ($(Q10)+(-1.35,-1)$) node {\footnotesize $0$};
\draw ($(Q10)+(-0.5,-1)$) node {\footnotesize $0$};
\draw ($(Q10)+(0.5,-1)$) node {\footnotesize $1$};
\draw ($(Q10)+(1.35,-1)$) node {\footnotesize $1$};
\draw ($(Q10)+(1.5,-0.675)$) node {\footnotesize $0$};

\draw ($(Q10)+(0.4,0)$) node {\footnotesize $0$};
\draw ($(Q10)+(0.2,0.3)$) node {\footnotesize $1$};
\draw ($(Q10)+(0.2,-0.3)$) node {\footnotesize $1$};
\draw ($(Q10)+(-0.4,0)$) node {\footnotesize $1$};
\draw ($(Q10)+(-0.2,0.3)$) node {\footnotesize $0$};
\draw ($(Q10)+(-0.2,-0.3)$) node {\footnotesize $0$};

\draw (3,-11.5) node[right] {$y_1=1-x_1+x_2-x_3$}; 
\draw (3,-12) node[right] {$y_2=x_3-x_4+x_5$}; 
\draw (3,-12.5) node[right] {$y_3=x_1-x_5+x_6$}; 
\draw (3,-13) node[right] {$y_4=1-x_2+x_4-x_6=2-(y_1+y_2+y_3)$}; 

\end{tikzpicture}
\caption{Heaviside functions $x_1, \dots, x_6$ and 
generalized Heaviside functions $y_1, \dots, y_4$.}
\label{fig:A3gHeav}
\end{figure}
For a chamber $C_i$, 
the product $x_p\cdot 1_{C_i}$ (or $y_p\cdot 1_{C_i}$) equals 
either $1_{C_i}$ or $0$; 
the full product table appears in Figure \ref{fig:prodtable}. 
(where, we put $1$ if the product is $1_{C_i}$). 
\begin{figure}
{\tiny 
\begin{tabular}{c|cccccccccccccccccccccccc}
&
$1$&$2$&$3$&$4$&$5$&$6$&
$1'$&$2'$&$3'$&$4'$&$5'$&$6'$&
$7$&$8$&$9$&$10$&$11$&$12$&
$7'$&$8'$&$9'$&$10'$&$11'$&$12'$\\
\hline
$x_1$&
$0$&$0$&$0$&$0$&$0$&$0$&
$1$&$1$&$1$&$1$&$1$&$1$&
$0$&$0$&$0$&$0$&$0$&$0$&
$1$&$1$&$1$&$1$&$1$&$1$\\
$x_2$&
$1$&$0$&$0$&$0$&$0$&$0$&
$0$&$1$&$1$&$1$&$1$&$1$&
$1$&$0$&$0$&$0$&$1$&$1$&
$0$&$1$&$1$&$1$&$0$&$0$\\
$x_3$&
$1$&$1$&$0$&$0$&$0$&$0$&
$0$&$0$&$1$&$1$&$1$&$1$&
$1$&$1$&$1$&$1$&$1$&$1$&
$0$&$0$&$0$&$0$&$0$&$0$\\
$x_4$&
$1$&$1$&$1$&$0$&$0$&$0$&
$0$&$0$&$0$&$1$&$1$&$1$&
$1$&$1$&$0$&$0$&$0$&$1$&
$0$&$0$&$1$&$1$&$1$&$0$\\
$x_5$&
$1$&$1$&$1$&$1$&$0$&$0$&
$0$&$0$&$0$&$0$&$1$&$1$&
$0$&$0$&$0$&$0$&$0$&$0$&
$1$&$1$&$1$&$1$&$1$&$1$\\
$x_6$&
$1$&$1$&$1$&$1$&$1$&$0$&
$0$&$0$&$0$&$0$&$0$&$1$&
$1$&$1$&$1$&$0$&$0$&$0$&
$0$&$0$&$0$&$1$&$1$&$1$\\
$y_1$&
$1$&$0$&$1$&$1$&$1$&$1$&
$0$&$1$&$0$&$0$&$0$&$0$&
$1$&$0$&$0$&$0$&$1$&$1$&
$0$&$1$&$1$&$1$&$0$&$0$\\
$y_2$&
$1$&$1$&$0$&$1$&$0$&$0$&
$0$&$0$&$1$&$0$&$1$&$1$&
$0$&$0$&$1$&$1$&$1$&$0$&
$1$&$1$&$0$&$0$&$0$&$1$\\
$y_3$&
$0$&$0$&$0$&$0$&$1$&$0$&
$1$&$1$&$1$&$1$&$0$&$1$&
$1$&$1$&$1$&$0$&$0$&$0$&
$0$&$0$&$0$&$1$&$1$&$1$\\
$y_4$&
$0$&$1$&$1$&$0$&$0$&$1$&
$1$&$0$&$0$&$1$&$1$&$0$&
$0$&$1$&$0$&$1$&$0$&$1$&
$1$&$0$&$1$&$0$&$1$&$0$
\end{tabular}
}
\caption{Product table between generalized Heaviside functions and 
$1_{C_i}$.}
\label{fig:prodtable}
\end{figure}
In Figure \ref{fig:genTopegraphs}, 
we depict several generalized tope graphs. 
First, $\gTope(x_1, x_2, x_3, x_4, x_5, x_6)$ 
is the original tope graph $\Tope(\A)$. 
Although $y_1, y_2, y_3$ are not Heaviside functions, 
$\gTope(x_1, x_3, x_5, y_1, y_2, y_3)$ is a tope graph 
of an oriented matroid and is isomorphic to $\Tope(\A)$. 
On the other hand, $\gTope(y_1, x_2, x_3, x_4, x_5, x_6)$ and 
$\gTope(x_1, x_2, x_4, y_1, y_2, y_3)$ are not a tope graphs 
of oriented matroids, because they have vertices of degree $2$.

\begin{figure}[htbp]
\centering
\begin{tikzpicture}[scale=0.9]

\coordinate (P0) at (0,0); 
\coordinate (P1) at ($(P0)+(5.5,3)$); 
\coordinate (P2) at ($(P0)+(4.5,5)$); 
\coordinate (P3) at ($(P0)+(4,6)$); 
\coordinate (P4) at ($(P0)+(3,6)$); 
\coordinate (P5) at ($(P0)+(2.5,5)$); 
\coordinate (P6) at ($(P0)+(1.5,3)$); 
\coordinate (P7) at ($(P0)+(1,2)$); 
\coordinate (P8) at ($(P0)+(1.5,1)$); 
\coordinate (P9) at ($(P0)+(2.5,1)$); 
\coordinate (P10) at ($(P0)+(4.5,1)$); 
\coordinate (P11) at ($(P0)+(5.5,1)$); 
\coordinate (P12) at ($(P0)+(6,2)$); 
\coordinate (P13) at ($(P0)+(4.5,3)$); 
\coordinate (P14) at ($(P0)+(4,4)$); 
\coordinate (P15) at ($(P0)+(3,4)$); 
\coordinate (P16) at ($(P0)+(2.5,3)$); 
\coordinate (P17) at ($(P0)+(3,2)$); 
\coordinate (P18) at ($(P0)+(4,2)$); 
\coordinate (P19) at ($(P0)+(0,2)$); 
\coordinate (P20) at ($(P0)+(1,0)$); 
\coordinate (P21) at ($(P0)+(6,0)$); 
\coordinate (P22) at ($(P0)+(7,2)$); 
\coordinate (P23) at ($(P0)+(4.5,7)$); 
\coordinate (P24) at ($(P0)+(2.5,7)$); 

\draw [thick] (P1)--(P2)--(P3)--(P4)--(P5)--(P6)--(P7)--(P8)--(P9)--(P10)--(P11)--(P12)--cycle;
\draw [thick] (P13)--(P14)--(P15)--(P16)--(P17)--(P18)--cycle;
\draw [thick] (P19)--(P20)-- node[below] {$\gTope(x_1, x_2, x_3, x_4, x_5, x_6)$} (P21)--(P22)--(P23)--(P24)--cycle;

\draw[thick] (P1)--(P13);
\draw[thick] (P2)--(P14);
\draw[thick] (P5)--(P15);
\draw[thick] (P6)--(P16);
\draw[thick] (P9)--(P17);
\draw[thick] (P10)--(P18);

\draw[thick] (P3)--(P23);
\draw[thick] (P4)--(P24);
\draw[thick] (P7)--(P19);
\draw[thick] (P8)--(P20);
\draw[thick] (P11)--(P21);
\draw[thick] (P12)--(P22);

\filldraw[draw=black, fill=white] (P1) node {\small $1$} circle [radius=0.2]; 
\filldraw[draw=black, fill=white] (P2) node {\small $2$} circle [radius=0.2]; 
\filldraw[draw=black, fill=white] (P3) node {\small $3$} circle [radius=0.2]; 
\filldraw[draw=black, fill=white] (P4) node {\small $4$} circle [radius=0.2]; 
\filldraw[draw=black, fill=white] (P5) node {\small $5$} circle [radius=0.2]; 
\filldraw[draw=black, fill=white] (P6) node {\small $6$} circle [radius=0.2]; 
\filldraw[draw=black, fill=white] (P7) node {\small $1'$} circle [radius=0.2]; 
\filldraw[draw=black, fill=white] (P8) node {\small $2'$} circle [radius=0.2]; 
\filldraw[draw=black, fill=white] (P9) node {\small $3'$} circle [radius=0.2]; 
\filldraw[draw=black, fill=white] (P10) node {\small $4'$} circle [radius=0.2]; 
\filldraw[draw=black, fill=white] (P11) node {\small $5'$} circle [radius=0.2]; 
\filldraw[draw=black, fill=white] (P12) node {\small $6'$} circle [radius=0.2]; 
\filldraw[draw=black, fill=white] (P13) node {\small $7$} circle [radius=0.2]; 
\filldraw[draw=black, fill=white] (P14) node {\small $8$} circle [radius=0.2]; 
\filldraw[draw=black, fill=white] (P15) node {\small $9$} circle [radius=0.2]; 
\filldraw[draw=black, fill=white] (P16) node {\small $10$} circle [radius=0.2]; 
\filldraw[draw=black, fill=white] (P17) node {\small $11$} circle [radius=0.2]; 
\filldraw[draw=black, fill=white] (P18) node {\small $12$} circle [radius=0.2]; 
\filldraw[draw=black, fill=white] (P19) node {\small $7'$} circle [radius=0.2]; 
\filldraw[draw=black, fill=white] (P20) node {\small $8'$} circle [radius=0.2]; 
\filldraw[draw=black, fill=white] (P21) node {\small $9'$} circle [radius=0.2]; 
\filldraw[draw=black, fill=white] (P22) node {\footnotesize $10'$} circle [radius=0.2]; 
\filldraw[draw=black, fill=white] (P23) node {\footnotesize $11'$} circle [radius=0.2]; 
\filldraw[draw=black, fill=white] (P24) node {\footnotesize $12'$} circle [radius=0.2];

\coordinate (Q0) at (8,0); 
\coordinate (Q1) at ($(Q0)+(5.5,3)$); 
\coordinate (Q2) at ($(Q0)+(4.5,5)$); 
\coordinate (Q3) at ($(Q0)+(4,6)$); 
\coordinate (Q4) at ($(Q0)+(3,6)$); 
\coordinate (Q5) at ($(Q0)+(2.5,5)$); 
\coordinate (Q6) at ($(Q0)+(1.5,3)$); 
\coordinate (Q7) at ($(Q0)+(1,2)$); 
\coordinate (Q8) at ($(Q0)+(1.5,1)$); 
\coordinate (Q9) at ($(Q0)+(2.5,1)$); 
\coordinate (Q10) at ($(Q0)+(4.5,1)$); 
\coordinate (Q11) at ($(Q0)+(5.5,1)$); 
\coordinate (Q12) at ($(Q0)+(6,2)$); 
\coordinate (Q13) at ($(Q0)+(4.5,3)$); 
\coordinate (Q14) at ($(Q0)+(4,4)$); 
\coordinate (Q15) at ($(Q0)+(3,4)$); 
\coordinate (Q16) at ($(Q0)+(2.5,3)$); 
\coordinate (Q17) at ($(Q0)+(3,2)$); 
\coordinate (Q18) at ($(Q0)+(4,2)$); 
\coordinate (Q19) at ($(Q0)+(0,2)$); 
\coordinate (Q20) at ($(Q0)+(1,0)$); 
\coordinate (Q21) at ($(Q0)+(6,0)$); 
\coordinate (Q22) at ($(Q0)+(7,2)$); 
\coordinate (Q23) at ($(Q0)+(4.5,7)$); 
\coordinate (Q24) at ($(Q0)+(2.5,7)$); 

\draw [thick] (Q1)--(Q2)--(Q3)--(Q4)--(Q5)--(Q6)--(Q7)--(Q8)--(Q9)--(Q10)--(Q11)--(Q12)--cycle;
\draw [thick] (Q13)--(Q14)--(Q15)--(Q16)--(Q17)--(Q18)--cycle;
\draw [thick] (Q19)--(Q20)-- node[below] {$\gTope(x_1, x_3, x_5, y_1, y_2, y_3)$} (Q21)--(Q22)--(Q23)--(Q24)--cycle;

\draw[thick] (Q1)--(Q13);
\draw[thick] (Q2)--(Q14);
\draw[thick] (Q5)--(Q15);
\draw[thick] (Q6)--(Q16);
\draw[thick] (Q9)--(Q17);
\draw[thick] (Q10)--(Q18);

\draw[thick] (Q3)--(Q23);
\draw[thick] (Q4)--(Q24);
\draw[thick] (Q7)--(Q19);
\draw[thick] (Q8)--(Q20);
\draw[thick] (Q11)--(Q21);
\draw[thick] (Q12)--(Q22);

\filldraw[draw=black, fill=white] (Q1) node {\small $2$} circle [radius=0.2]; 
\filldraw[draw=black, fill=white] (Q2) node {\small $1$} circle [radius=0.2]; 
\filldraw[draw=black, fill=white] (Q3) node {\small $4$} circle [radius=0.2]; 
\filldraw[draw=black, fill=white] (Q4) node {\small $3$} circle [radius=0.2]; 
\filldraw[draw=black, fill=white] (Q5) node {\small $6$} circle [radius=0.2]; 
\filldraw[draw=black, fill=white] (Q6) node {\small $5$} circle [radius=0.2]; 
\filldraw[draw=black, fill=white] (Q7) node {\small $2'$} circle [radius=0.2]; 
\filldraw[draw=black, fill=white] (Q8) node {\small $1'$} circle [radius=0.2]; 
\filldraw[draw=black, fill=white] (Q9) node {\small $4'$} circle [radius=0.2]; 
\filldraw[draw=black, fill=white] (Q10) node {\small $3'$} circle [radius=0.2]; 
\filldraw[draw=black, fill=white] (Q11) node {\small $6'$} circle [radius=0.2]; 
\filldraw[draw=black, fill=white] (Q12) node {\small $5'$} circle [radius=0.2]; 
\filldraw[draw=black, fill=white] (Q13) node {\small $10$} circle [radius=0.2]; 
\filldraw[draw=black, fill=white] (Q14) node {\small $11$} circle [radius=0.2]; 
\filldraw[draw=black, fill=white] (Q15) node {\small $12$} circle [radius=0.2]; 
\filldraw[draw=black, fill=white] (Q16) node {\small $7$} circle [radius=0.2]; 
\filldraw[draw=black, fill=white] (Q17) node {\small $8$} circle [radius=0.2]; 
\filldraw[draw=black, fill=white] (Q18) node {\small $9$} circle [radius=0.2]; 
\filldraw[draw=black, fill=white] (Q19) node {\footnotesize $10'$} circle [radius=0.2]; 
\filldraw[draw=black, fill=white] (Q20) node {\footnotesize $11'$} circle [radius=0.2]; 
\filldraw[draw=black, fill=white] (Q21) node {\footnotesize $12'$} circle [radius=0.2]; 
\filldraw[draw=black, fill=white] (Q22) node {\small $7'$} circle [radius=0.2]; 
\filldraw[draw=black, fill=white] (Q23) node {\small $8'$} circle [radius=0.2]; 
\filldraw[draw=black, fill=white] (Q24) node {\small $9'$} circle [radius=0.2];

\coordinate (R0) at (0,-8.5); 
\coordinate (R1) at ($(R0)+(5.5,3)$); 
\coordinate (R2) at ($(R0)+(4.5,5)$); 
\coordinate (R3) at ($(R0)+(4,6)$); 
\coordinate (R4) at ($(R0)+(3,6)$); 
\coordinate (R5) at ($(R0)+(2.5,5)$); 
\coordinate (R6) at ($(R0)+(1.5,3)$); 
\coordinate (R7) at ($(R0)+(1,2)$); 
\coordinate (R8) at ($(R0)+(1.5,1)$); 
\coordinate (R9) at ($(R0)+(2.5,1)$); 
\coordinate (R10) at ($(R0)+(4.5,1)$); 
\coordinate (R11) at ($(R0)+(5.5,1)$); 
\coordinate (R12) at ($(R0)+(6,2)$); 
\coordinate (R13) at ($(R0)+(4.5,3)$); 
\coordinate (R14) at ($(R0)+(4,4)$); 
\coordinate (R15) at ($(R0)+(3,4)$); 
\coordinate (R16) at ($(R0)+(2.5,3)$); 
\coordinate (R17) at ($(R0)+(3,2)$); 
\coordinate (R18) at ($(R0)+(4,2)$); 
\coordinate (R19) at ($(R0)+(0,2)$); 
\coordinate (R20) at ($(R0)+(1,0)$); 
\coordinate (R21) at ($(R0)+(6,0)$); 
\coordinate (R22) at ($(R0)+(7,2)$); 
\coordinate (R23) at ($(R0)+(4.5,7)$); 
\coordinate (R24) at ($(R0)+(2.5,7)$); 

\draw [thick] (R1)--(R2)--(R3)--(R4)--(R5)--(R6)--(R7)--(R8)--(R9)--(R10)--(R11)--(R12)--cycle;
\draw [thick] (R13)--(R14)--(R15)--(R16)--(R17)--(R18)--cycle;
\draw [thick] (R19)--(R20)-- node[below] {$\gTope(y_1, x_2, x_3, x_4, x_5, x_6)$} (R21)--(R22)--(R23)--(R24)--cycle;

\draw[thick] (R2)--(R14);
\draw[thick] (R5)--(R15);
\draw[thick] (R9)--(R17);
\draw[thick] (R10)--(R18);

\draw[thick] (R3)--(R23);
\draw[thick] (R4)--(R24);
\draw[thick] (R8)--(R20);
\draw[thick] (R11)--(R21);

\filldraw[draw=black, fill=white] (R1) node {\small $8$} circle [radius=0.2]; 
\filldraw[draw=black, fill=white] (R2) node {\small $2$} circle [radius=0.2]; 
\filldraw[draw=black, fill=white] (R3) node {\footnotesize $11'$} circle [radius=0.2]; 
\filldraw[draw=black, fill=white] (R4) node {\small $3$} circle [radius=0.2]; 
\filldraw[draw=black, fill=white] (R5) node {\footnotesize $10'$} circle [radius=0.2]; 
\filldraw[draw=black, fill=white] (R6) node {\small $9'$} circle [radius=0.2]; 
\filldraw[draw=black, fill=white] (R7) node {\small $8'$} circle [radius=0.2]; 
\filldraw[draw=black, fill=white] (R8) node {\small $2'$} circle [radius=0.2]; 
\filldraw[draw=black, fill=white] (R9) node {\small $11$} circle [radius=0.2]; 
\filldraw[draw=black, fill=white] (R10) node {\small $3'$} circle [radius=0.2]; 
\filldraw[draw=black, fill=white] (R11) node {\small $10$} circle [radius=0.2]; 
\filldraw[draw=black, fill=white] (R12) node {\small $9$} circle [radius=0.2]; 
\filldraw[draw=black, fill=white] (R13) node {\small $5'$} circle [radius=0.2]; 
\filldraw[draw=black, fill=white] (R14) node {\small $6'$} circle [radius=0.2]; 
\filldraw[draw=black, fill=white] (R15) node {\small $1$} circle [radius=0.2]; 
\filldraw[draw=black, fill=white] (R16) node {\small $7$} circle [radius=0.2]; 
\filldraw[draw=black, fill=white] (R17) node {\small $12$} circle [radius=0.2]; 
\filldraw[draw=black, fill=white] (R18) node {\small $4'$} circle [radius=0.2]; 
\filldraw[draw=black, fill=white] (R19) node {\small $5$} circle [radius=0.2]; 
\filldraw[draw=black, fill=white] (R20) node {\small $6$} circle [radius=0.2]; 
\filldraw[draw=black, fill=white] (R21) node {\small $1'$} circle [radius=0.2]; 
\filldraw[draw=black, fill=white] (R22) node {\small $7'$} circle [radius=0.2]; 
\filldraw[draw=black, fill=white] (R23) node {\footnotesize $12'$} circle [radius=0.2]; 
\filldraw[draw=black, fill=white] (R24) node {\small $4$} circle [radius=0.2];

\coordinate (S0) at (11.5,-5); 
\coordinate (S1) at ($(S0)+(-1,2)$); 
\coordinate (S2) at ($(S0)+(-2,-3)$); 
\coordinate (S3) at ($(S0)+(0,1)$); 
\coordinate (S4) at ($(S0)+(-2,1)$); 
\coordinate (S5) at ($(S0)+(-1,0)$); 
\coordinate (S6) at ($(S0)+(-1,1)$); 
\coordinate (S7) at ($(S0)+(1,1)$); 
\coordinate (S8) at ($(S0)+(2,-3)$); 
\coordinate (S9) at ($(S0)+(-1,-3)$); 
\coordinate (S10) at ($(S0)+(-2,-2)$); 
\coordinate (S11) at ($(S0)+(-2,2)$); 
\coordinate (S12) at ($(S0)+(0,2)$); 

\coordinate (S01) at ($(S0)+(1,-2)$); 
\coordinate (S02) at ($(S0)+(2,3)$); 
\coordinate (S03) at ($(S0)+(0,-1)$); 
\coordinate (S04) at ($(S0)+(2,-1)$); 
\coordinate (S05) at ($(S0)+(1,0)$); 
\coordinate (S06) at ($(S0)+(1,-1)$); 
\coordinate (S07) at ($(S0)+(-1,-1)$); 
\coordinate (S08) at ($(S0)+(-2,3)$); 
\coordinate (S09) at ($(S0)+(1,3)$); 
\coordinate (S010) at ($(S0)+(2,2)$); 
\coordinate (S011) at ($(S0)+(2,-2)$); 
\coordinate (S012) at ($(S0)+(0,-2)$); 

\draw [thick] (S2)--(S08);
\draw [thick] (S10)--(S07)--(S012)--(S9)--cycle;
\draw [thick] (S012)--(S03)--(S04)--(S011)--cycle;
\draw [thick] (S05)--(S06);

\draw [thick] (S02)--(S8);
\draw [thick] (S010)--(S7)--(S12)--(S09)--cycle;
\draw [thick] (S12)--(S3)--(S4)--(S11)--cycle;
\draw [thick] (S5)--(S6);

\draw [thick] ($(S0)+(0,-3.5)$) node[below] {$\gTope(x_1, x_2, x_4, y_1, y_2, y_3)$};

\filldraw[draw=black, fill=white] (S1) node {\small $1$} circle [radius=0.2]; 
\filldraw[draw=black, fill=white] (S2) node {\small $2$} circle [radius=0.2]; 
\filldraw[draw=black, fill=white] (S3) node {\small $3$} circle [radius=0.2]; 
\filldraw[draw=black, fill=white] (S4) node {\small $4$} circle [radius=0.2]; 
\filldraw[draw=black, fill=white] (S5) node {\small $5$} circle [radius=0.2]; 
\filldraw[draw=black, fill=white] (S6) node {\small $6$} circle [radius=0.2]; 
\filldraw[draw=black, fill=white] (S7) node {\small $7$} circle [radius=0.2]; 
\filldraw[draw=black, fill=white] (S8) node {\small $8$} circle [radius=0.2]; 
\filldraw[draw=black, fill=white] (S9) node {\small $9$} circle [radius=0.2]; 
\filldraw[draw=black, fill=white] (S10) node {\small $10$} circle [radius=0.2]; 
\filldraw[draw=black, fill=white] (S11) node {\small $11$} circle [radius=0.2]; 
\filldraw[draw=black, fill=white] (S12) node {\small $12$} circle [radius=0.2]; 

\filldraw[draw=black, fill=white] (S01) node {\small $1'$} circle [radius=0.2]; 
\filldraw[draw=black, fill=white] (S02) node {\small $2'$} circle [radius=0.2]; 
\filldraw[draw=black, fill=white] (S03) node {\small $3'$} circle [radius=0.2]; 
\filldraw[draw=black, fill=white] (S04) node {\small $4'$} circle [radius=0.2]; 
\filldraw[draw=black, fill=white] (S05) node {\small $5'$} circle [radius=0.2]; 
\filldraw[draw=black, fill=white] (S06) node {\small $6'$} circle [radius=0.2]; 
\filldraw[draw=black, fill=white] (S07) node {\small $7'$} circle [radius=0.2]; 
\filldraw[draw=black, fill=white] (S08) node {\small $8'$} circle [radius=0.2]; 
\filldraw[draw=black, fill=white] (S09) node {\small $9'$} circle [radius=0.2]; 
\filldraw[draw=black, fill=white] (S010) node {\footnotesize $10'$} circle [radius=0.2]; 
\filldraw[draw=black, fill=white] (S011) node {\footnotesize $11'$} circle [radius=0.2]; 
\filldraw[draw=black, fill=white] (S012) node {\footnotesize $12'$} circle [radius=0.2];

\filldraw[draw=black, fill=white] (R1) node {\small $8$} circle [radius=0.2]; 
\filldraw[draw=black, fill=white] (R2) node {\small $2$} circle [radius=0.2]; 
\filldraw[draw=black, fill=white] (R3) node {\footnotesize $11'$} circle [radius=0.2]; 
\filldraw[draw=black, fill=white] (R4) node {\small $3$} circle [radius=0.2]; 
\filldraw[draw=black, fill=white] (R5) node {\footnotesize $10'$} circle [radius=0.2]; 
\filldraw[draw=black, fill=white] (R6) node {\small $9'$} circle [radius=0.2]; 
\filldraw[draw=black, fill=white] (R7) node {\small $8'$} circle [radius=0.2]; 
\filldraw[draw=black, fill=white] (R8) node {\small $2'$} circle [radius=0.2]; 
\filldraw[draw=black, fill=white] (R9) node {\small $11$} circle [radius=0.2]; 
\filldraw[draw=black, fill=white] (R10) node {\small $3'$} circle [radius=0.2]; 
\filldraw[draw=black, fill=white] (R11) node {\small $10$} circle [radius=0.2]; 
\filldraw[draw=black, fill=white] (R12) node {\small $9$} circle [radius=0.2]; 
\filldraw[draw=black, fill=white] (R13) node {\small $5'$} circle [radius=0.2]; 
\filldraw[draw=black, fill=white] (R14) node {\small $6'$} circle [radius=0.2]; 
\filldraw[draw=black, fill=white] (R15) node {\small $1$} circle [radius=0.2]; 
\filldraw[draw=black, fill=white] (R16) node {\small $7$} circle [radius=0.2]; 
\filldraw[draw=black, fill=white] (R17) node {\small $12$} circle [radius=0.2]; 
\filldraw[draw=black, fill=white] (R18) node {\small $4'$} circle [radius=0.2]; 
\filldraw[draw=black, fill=white] (R19) node {\small $5$} circle [radius=0.2]; 
\filldraw[draw=black, fill=white] (R20) node {\small $6$} circle [radius=0.2]; 
\filldraw[draw=black, fill=white] (R21) node {\small $1'$} circle [radius=0.2]; 
\filldraw[draw=black, fill=white] (R22) node {\small $7'$} circle [radius=0.2]; 
\filldraw[draw=black, fill=white] (R23) node {\footnotesize $12'$} circle [radius=0.2]; 
\filldraw[draw=black, fill=white] (R24) node {\small $4$} circle [radius=0.2];

\end{tikzpicture}
\caption{Some generalized tope graphs.}
\label{fig:genTopegraphs}
\end{figure}
\end{example}

The point of Conjecture \ref{conj:algorithm} is that even when 
$y_1, \dots, y_n$ are not Heaviside functions, if the 
resulting graph is an oriented matroid tope graph, then it must 
be isomorphic to $\Tope(\A)$.

\section{From graded VG algebras to signed circuits}
\label{sec:vgsign}

\subsection{Square-zero elements}
\label{subsec:sq0}

In the graded VG algebra, the square of a Heaviside class is zero. 
Thus, instead of idempotents, we consider the set of 
square-zero elements: 
\begin{equation}
\label{eq:setsq0}
\sqzero(\A):=\{\overline{y}\in\grVG^1(\A)\mid \overline{y}^2=0 \mbox{ in }\grVG^2(\A)\}. 
\end{equation}
If $y\in\gHeav(\A)\subset\Fil^1\cap\Idem(\A)$, then $y^2=y\in\Fil^1$, 
hence $\overline{y}\in\sqzero(\A)$. 
We now describe $\sqzero(\A)$ in terms of 
generalized Heaviside functions. 
\begin{lemma}
\label{lem:sq0}
Assume $\cha R\neq 2$. Then, 
\begin{equation}
\label{eq:sq0}
\sqzero(\A)=
\bigcup_{y\in\gHeav(\A)}R\cdot \overline{y}. 
\end{equation}
\end{lemma}
\begin{proof}
Since $\grVG^1=\Fil^1/\Fil^0$, the set of square-zero elements 
can be written as 
\begin{equation}
\sqzero(\A)=
\{y\in\Fil^1\VG(\A)\mid y^2\in\Fil^1\VG(\A)\}/\Fil^0. 
\end{equation}
Now the result follows from Theorem \ref{thm:genHeav}. 
\end{proof}
As an application of Lemma \ref{lem:sq0}, we can distinguish 
non-isomorphic graded VG algebras even when the graded 
$R$-module structures agree. 
\begin{example}
\label{ex:disting}
Assume $\cha R\neq 2$. 
Let $\A_1$ and $\A_2$ be arrangements of six planes in $\R^3$ 
as in Figure \ref{fig:disting}. 
Their characteristic polynomials are 
$\chi(\A_1, t)=\chi(\A_2, t)=(t-1)(t-2)(t-3)$, so, 
$\grVG^\bullet(\A_1)$ and 
$\grVG^\bullet(\A_2)$ are isomorphic as graded $R$-modules. 
It is known that 
$\OS^\bullet(\A_1)\not\simeq\OS^\bullet(\A_2)$ 
\cite[Example 3.1]{fal-alg}. 
Moreover, the graded VG algebras are also non-isomorphic 
$\grVG^\bullet(\A_1)\not\simeq\grVG^\bullet(\A_2)$. 
Here we write $x_i$ for the class $\overline{x_i^+}\in\grVG^1(\A)$.
Then Lemma \ref{lem:sq0} gives 
\begin{equation}
\begin{aligned}
\sqzero(\A_1)=
&\bigcup_{i=1}^6Rx_i
\cup R(x_1+x_2-x_3)
\cup R(x_1-x_4+x_5)\\
&\cup R(x_1-x_4+x_6)
\cup R(x_1-x_5+x_6)
\cup R(x_4-x_5+x_6)
\\
\sqzero(\A_2)=
&\bigcup_{i=1}^6Rx_i
\cup R(x_1+x_2-x_3)
\cup R(x_1+x_5-x_6)\\
&\cup R(x_2-x_4-x_6)
\cup R(x_3-x_4-x_5).
\end{aligned}
\end{equation}
The numbers of components are different, hence 
the graded VG algebras can not be isomorphic. 
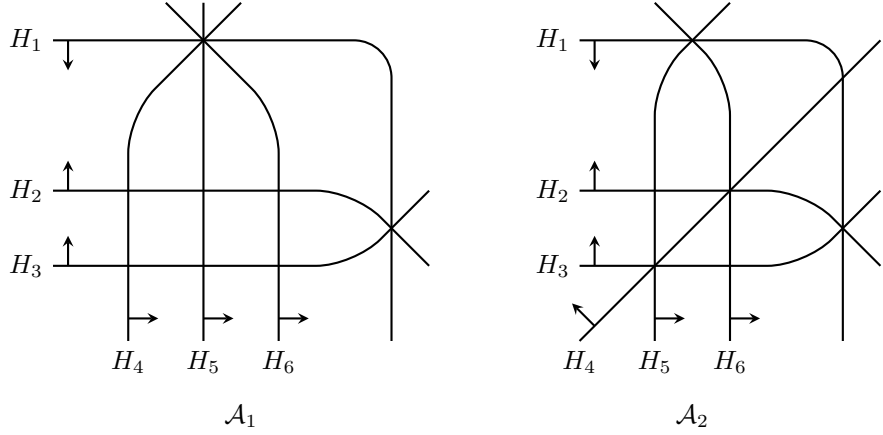
\begin{figure}[htbp]
\centering
\begin{tikzpicture}

\draw[rounded corners=5mm, thick] (0,4) node[left] {$H_1$} --(4.5,4)--(4.5,0);
\draw[rounded corners=5mm, thick] (1.5,4.5)--(3,3)--(3,0) node[below] {$H_6$};
\draw[rounded corners=5mm, thick] (2.5,4.5)--(1,3)--(1,0) node[below] {$H_4$} ;
\draw[rounded corners=5mm, thick] (2,4.5)--(2,0) node[below] {$H_5$};
\draw[rounded corners=5mm, thick] (0,2) node[left] {$H_2$} --(4,2)--(5,1);
\draw[rounded corners=5mm, thick] (0,1) node[left] {$H_3$} --(4,1)--(5,2);
\draw (2.5,-1) node {$\A_1$};

\draw[-stealth, thick] (0.2,4) -- ++(0,-0.4);
\draw[-stealth, thick] (0.2,2) -- ++(0,0.4);
\draw[-stealth, thick] (0.2,1) -- ++(0,0.4);
\draw[-stealth, thick] (1,0.3) -- ++(0.4,0);
\draw[-stealth, thick] (2,0.3) -- ++(0.4,0);
\draw[-stealth, thick] (3,0.3) -- ++(0.4,0);

\draw[rounded corners=5mm, thick] (7,4) node[left] {$H_1$} --(10.5,4)--(10.5,0);
\draw[rounded corners=5mm, thick] (9,4.5)--(8,3.5)--(8,0) node[below] {$H_5$};
\draw[rounded corners=5mm, thick] (8,4.5)--(9,3.5)--(9,0) node[below] {$H_6$};
\draw[rounded corners=5mm, thick] (7,1) node[left] {$H_3$}--(10,1)--(11,2);
\draw[rounded corners=5mm, thick] (7,2) node[left] {$H_2$}--(10,2)--(11,1);
\draw[rounded corners=5mm, thick] (7,0) node[below] {$H_4$}--(11,4);
\draw (8.5,-1) node {$\A_2$};

\draw[-stealth, thick] (7.2,4) -- ++(0,-0.4);
\draw[-stealth, thick] (7.2,2) -- ++(0,0.4);
\draw[-stealth, thick] (7.2,1) -- ++(0,0.4);
\draw[-stealth, thick] (7.2,0.2) -- ++(-0.3,0.3);
\draw[-stealth, thick] (8,0.3) -- ++(0.4,0);
\draw[-stealth, thick] (9,0.3) -- ++(0.4,0);

\end{tikzpicture}
\caption{$\A_1$ and $\A_2$ with the same characteristic polynomials}
\label{fig:disting}
\end{figure}
\end{example}

\subsection{Recovering signed circuits}
\label{subsec:signedcir}

\begin{theorem}
\label{thm:main2}
Assume $\cha R\neq 2$. 
Let $\A_1, \A_2$ be real arrangements and suppose 
that $\A_1$ is generic in codimension $2$. 
If $\grVG(\A_1)$ and $\grVG(\A_2)$ are isomorphic as 
graded $R$-algebras, then the set of signed circuits 
$\circuit(\A_1)$ and $\circuit(\A_2)$ are isomorphic. 
\end{theorem}

The idea of the proof is as follows. By the genericity assumption, 
the set of square-zero elements are exactly scalar multiples 
of Heaviside functions in $\grVG^1$. A refined version of Cordovil's 
argument \cite[Proposition 3.4]{cor-com} enables us to 
recover signed circuits. 

\begin{proof}[Proof of Theorem \ref{thm:main2}]
Let $\A=\{H_1, \dots, H_n\}$ be a real arrangement in $\R^\ell$, and 
let $\alpha_i\in V^*$ be a defining equation of $H_i$. 
Assume $\A$ is generic in codimension $2$. 
We explain how to reconstruct signed circuits 
$\circuit(\A)$ from the graded VG algebra $\grVG^\bullet(\A)$. 

By Lemma \ref{lem:equal} we have $\gHeav(\A)=\Heav(\A)$, hence 
\begin{equation}
\sqzero(\A)=R\cdot f_1\cup\cdots\cup R\cdot f_n, 
\end{equation}
for some $f_1, \dots, f_n\in\grVG^1(\A)$. 
After permuting $[n]$ we may assume that 
\begin{equation}
R\cdot f_i=R\cdot \overline{x_i^+}, 
\end{equation}
as sets ($i=1, \dots, n$), 
where $x_i^+\in\Fil^1\VG(\A)$ is the Heaviside function 
for $H_i$. Hence, there exists $\mu_i\in R^\times$ such that 
$f_i=\mu_i\cdot\overline{x_i^+}$ ($i=1, \dots, n$). 

Let $I=\{i_1, \dots, i_k\}\subset [n]$. 
By Cordovil \cite[Corollary 2.5]{cor-com}, 
the linear forms 
$\alpha_{i_1}, \alpha_{i_2}, \dots, \alpha_{i_k}$ are 
independent if and only if 
\begin{equation}
\overline{x_{i_1}^+}\cdot\overline{x_{i_2}^+}\cdots
\overline{x_{i_k}^+}\neq 0, \mbox{ in }\ \grVG^k(\A), 
\end{equation}
equivalent to $f_{i_1}\cdots f_{i_k}\neq 0$. 
(For simplicity, write $f_I=f_{i_1}\cdots f_{i_k}$.) 
Thus, 
$\{\alpha_{i_1}, \dots, \alpha_{i_k}\}$ is a circuit if and only if 
\begin{equation}
f_I=0, \mbox{ and }
f_{I\setminus\{i_p\}}\neq 0
\mbox{ for $p=1, \dots,k$}, 
\end{equation}
where $I=\{i_1, \dots, i_k\}$. 
Hence the matroid structure (equivalently the intersection lattice 
$L(\A)$) of $\A$ is recovered from $f_1, \dots, f_n$. 

It remains to recover signed circuits. 
Let $I=\{i_1, \dots, i_k\}\subset[n]$ be a circuit. 
Then there is a non trivial linear relation 
\begin{equation}
\label{eq:linrel} 
\lambda^I_1\cdot f_{I\setminus\{i_1\}}+
\lambda^I_2\cdot f_{I\setminus\{i_2\}}+\cdots+
\lambda^I_k\cdot f_{I\setminus\{i_k\}}=0, 
\end{equation}
where $\lambda^I_1, \dots, \lambda^I_k\in R$, and necessarily 
$\lambda^I_p\neq 0$ ($p=1, \dots, k$). 
Such a tuple 
$(\lambda_1, \dots, \lambda_k)$ is unique up to 
nonzero scalar multiple by an element in the quotient field of $R$. 

Call generators $f_1, \dots, f_n$ \emph{good} if, 
for every circuit $I\subset[n]$, the relation \eqref{eq:linrel} can 
be chosen so that $\lambda^I_p=\pm 1$ for all $p=1, \dots, k$. 
By Proposition \ref{prop:presenGRVG}, the classes 
$\overline{x_1^+}, \dots, \overline{x_n^+}$ are good generators, 
therefore, good generators exist. 

Fix good generators $f_1, \dots, f_n$. 
For each circuit $I=\{i_1, \dots, i_k\}\subset [n]$ choose a  
linear relation 
$\sum_{p=1}^k\sigma_p\cdot f_{I\setminus\{i_p\}}=0$ 
with $\sigma_p\in\{\pm 1\}$. 
Define the sign vector 
$(\lambda_1, \dots, \lambda_n)\in\{0, \pm\}^n$ by 
\begin{equation}
\lambda_i=
\begin{cases}
\sigma_p, \ \mbox{$i\in I$ and $i=i_p$}\\
0, \ i\notin I. 
\end{cases}
\end{equation}
We call $(\lambda_1, \dots, \lambda_n)$ the extension of $(\sigma_p)$ 
by $0$. 
Now, let us define $\check{\circuit}(f_1, \dots, f_n)$ as 
\begin{equation}
\check{\circuit}(f_1, \dots, f_n)=
\left\{
(\lambda_1, \dots, \lambda_n)
\middle|\ 
\begin{lgathered}
\mbox{$\exists$ a circuit $I=\{i_1, \dots, i_k\}$ with a}\\
\mbox{relation }\sum_{p=1}^k\sigma_p\cdot f_{I\setminus\{i_p\}}=0
\mbox{ such that}\\
\mbox{$(\lambda_i)$ is the extension of $(\sigma_p)$ by $0$}
\end{lgathered}
\right\}.
\end{equation}
By Remark \ref{rem:uniqurel}, 
\begin{equation}
\check{\circuit}(\overline{x_1^+}, \dots, \overline{x_n^+})=
\circuit(\alpha_1, \dots, \alpha_n). 
\end{equation}
It remains to show that if 
$(f_1, \dots, f_n)$ and $(f'_1, \dots, f'_n)$ 
are good generators with $R\cdot f_i=R\cdot f'_i$ ($i=1, \dots, n$), 
then 
$\check{\circuit}(f_1, \dots, f_n)$ and 
$\check{\circuit}(f'_1, \dots, f'_n)$ differ 
by a reorientation. 

Assume that $(f_1, \dots, f_n)$ and $(f'_1, \dots, f'_n)$ 
are good generators. 
For each circuit $I=\{i_1, \dots, i_k\}\subset[n]$, 
we have relations 
\begin{equation}
\label{eq:f1}
\lambda_1^I\cdot f_{I\setminus\{i_1\}}+
\lambda_2^I\cdot f_{I\setminus\{i_2\}}+\cdots+
\lambda_k^I\cdot f_{I\setminus\{i_k\}}=0
\end{equation}
\begin{equation}
\label{eq:fprime}
\lambda_1^{'I}\cdot f_{I\setminus\{i_1\}}'+
\lambda_2^{'I}\cdot f_{I\setminus\{i_2\}}'+\cdots+
\lambda_k^{'I}\cdot f_{I\setminus\{i_k\}}'=0, 
\end{equation}
with 
$\lambda_{p}^I, \lambda_{p}^{'I}\in\{\pm 1\}$. 
Since $R\cdot f_i=R\cdot f'_i$, 
there exist $\mu_i\in R^\times$ such that $f_i=\mu_i\cdot f_i'$ 
($i=1, \dots, n$), hence 
\begin{equation}
\label{eq:fprime2}
\mu_{I\setminus\{i_1\}}\lambda_1^I\cdot f'_{I\setminus\{i_1\}}+
\mu_{I\setminus\{i_2\}}\lambda_2^I\cdot f'_{I\setminus\{i_2\}}+\cdots+
\mu_{I\setminus\{i_k\}}\lambda_k^I\cdot f'_{I\setminus\{i_k\}}=0. 
\end{equation}
Since \eqref{eq:fprime} and \eqref{eq:fprime2} represent the same 
relation up to a unit, 
there exists $c_I\in R^\times$ such that 
\begin{equation}
(\mu_{I\setminus\{i_1\}}\lambda_1^I, 
\mu_{I\setminus\{i_2\}}\lambda_2^I, \dots, 
\mu_{I\setminus\{i_k\}}\lambda_k^I)=
c_I\cdot(\lambda_1^{'I}, \lambda_2^{'I}, \dots, \lambda_k^{'I}). 
\end{equation}
Comparing $p$-th and $q$-th component ($1\leq p<q\leq k$), 
we have 
$\mu_{I\setminus\{i_p\}}\lambda_p^I=c_I\lambda_p^{'I}$ and 
$\mu_{I\setminus\{i_q\}}\lambda_q^I=c_I\lambda_q^{'I}$, 
and therefore, 
\begin{equation}
\mu_{I\setminus\{i_q\}}=
\frac{\lambda_p^I\cdot\lambda_q^{'I}}{\lambda_p^{'I}\cdot\lambda_q^I}\cdot\mu_{I\setminus\{i_p\}}
=\pm\mu_{I\setminus\{i_p\}}. 
\end{equation}
Since 
$\mu_{i_p}=\mu_I/\mu_{I\setminus\{i_p\}}$ and 
$\mu_{i_q}=\mu_I/\mu_{I\setminus\{i_q\}}$, we have 
$\mu_{i_p}=\pm\mu_{i_q}$. 
Thus, if indices $i, j\in [n]$ are contained in a single circuit 
$i, j\in I\subset[n]$, then $\mu_i=\pm\mu_j$. Let us define 
the equivalence relation $\sim$ on $[n]$ as the 
reflexive and transitive closure of 
the binary relation 
\begin{equation}
i\sim j\Longleftrightarrow
\mbox{there exists a circuit $I\subset[n]$ such that $i, j\in I$}. 
\end{equation}
Then, clearly, $i\sim j$ implies $\mu_i=\pm\mu_j$. 
Let $[n]=\bigsqcup_{\tau\in T}X_\tau$ be the partition into 
equivalence classes and choose a representative $r_\tau\in X_\tau$ 
from each equivalence class $X_\tau$. 
For each $i\in[n]$, there exists unique $\tau\in T$ 
such that $i\in X_\tau$. 
Now define $\widetilde{\mu}_i\in \{\pm 1\}$ by 
\begin{equation}
\widetilde{\mu}_i:=\frac{\mu_i}{\mu_{r_\tau}},\ \ (i\in X_\tau). 
\end{equation}
Now let $\widetilde{f}_i:=\widetilde{\mu}_i\cdot f_i'$. 
Then $\check{\circuit}(f'_1, \dots, f'_n)$ and 
$\check{\circuit}(\widetilde{f}_1, \dots, \widetilde{f}_n)$ 
are connected by the reorientation map 
$(\sigma_1, \dots, \sigma_n)\mapsto(\widetilde{\mu}_1\cdot\sigma_1, 
\dots, \widetilde{\mu}_n\cdot\sigma_n)$ of sign vectors. 
We will also see 
$\check{\circuit}(\widetilde{f}_1, \dots, \widetilde{f}_n)
=\check{\circuit}(f_1, \dots, f_n)$. 
Indeed, for any circuit 
$I=\{i_1, \dots, i_k\}\subset [n]$, there exists $\tau\in T$ 
such that $I\subset X_\tau$. 
And then 
\begin{equation}
\begin{aligned}
\lambda_1^I\cdot f_{I\setminus\{i_1\}}+
&\cdots+
\lambda_k^I\cdot f_{I\setminus\{i_k\}}\\
=&
\mu_{I\setminus\{i_1\}}\lambda_1^I\cdot f'_{I\setminus\{i_1\}}+
\cdots+
\mu_{I\setminus\{i_k\}}\lambda_k^I\cdot f'_{I\setminus\{i_k\}}\\
=&
\mu_{r_\tau}^{k-1}\cdot\widetilde{\mu}_{I\setminus\{i_1\}}\cdot
\lambda_1^I \cdot f'_{I\setminus\{i_1\}}+\cdots +
\mu_{r_\tau}^{k-1}\cdot\widetilde{\mu}_{I\setminus\{i_k\}}\cdot
\lambda_k^I \cdot f'_{I\setminus\{i_k\}}
\\
=&
\mu_{r_\tau}^{k-1}\cdot
(
\lambda_1^I\cdot \widetilde{f}_{I\setminus\{i_1\}}+
\cdots+
\lambda_k^I\cdot \widetilde{f}_{I\setminus\{i_k\}}
). 
\end{aligned}
\end{equation}
So, the linear relations among 
$f_{I\setminus\{i_1\}}, \dots, f_{I\setminus\{i_k\}}$ and those of 
$\widetilde{f}_{I\setminus\{i_1\}}, \dots, \widetilde{f}_{I\setminus\{i_k\}}$ are the same. 
Thus, we have 
$\check{\circuit}(f_1, \dots, f_n)=
\check{\circuit}(\widetilde{f}_1, \dots, \widetilde{f}_n)$. 
\end{proof}

\begin{example}
\label{ex:LtoGRVG}
Assume $\cha R\neq 2$. 
Let $\A_1$ and $\A_2$ be the generic $6$-plane arrangements 
from Example \ref{ex:6planes}. 
We have $L(\A_1)\simeq L(\A_2)$, hence 
$\OS^\bullet(\A_1)\simeq\OS^\bullet(\A_2)$ as 
graded $R$-algebras. 
However, graded VG algebras are not isomorphic, 
$\grVG^\bullet(\A_1)\not\simeq\grVG^\bullet(\A_2)$. 
Indeed, 
if $\grVG^\bullet(\A_1)\simeq\grVG^\bullet(\A_2)$, then 
by Theorem \ref{thm:main2}, the signed circuits would be equivalent 
(up to reorientation and relabeling), which would force the tope 
graphs to be isomorphic, contradicting Example \ref{ex:6planes}. 
Combined with Moseley's result \cite{mos-equ} mentioned in 
\S \ref{sec:intro}, we also have 
$H^\bullet(M_3(\A_1), R)\not\simeq H^\bullet(M_3(\A_2), R)$. 
In contrast to the case of $H^\bullet(M_2(\A), R)$, 
the cohomology ring $H^\bullet(M_3(\A), R)$ is not 
determined by the intersection lattice $L(\A)$. 
Note also that by Randell's lattice isotopy 
theorem \cite{ran-lat}, $M_2(\A_1)$ and $M_2(\A_2)$ are 
homeomorphic. 
In summary, neither $L(\A)$, $\OS^\bullet(\A)$, nor 
the space $M_2(\A)$ determines $\grVG^\bullet(\A)$ 
(equivalently, $H^\bullet(M_3(\A), R)$ as well). 
\end{example}

\begin{remark}
\label{rem:reconstgr}
Let $\A=\{H_1, \dots, H_n\}$ be a real 
arrangement in $\R^\ell$. In analogy with \S \ref{subsec:reconst}, 
we propose a conjectural 
algorithm recovering $\circuit(\A)$ from $\grVG^\bullet(\A)$ 
as follows. 
\begin{itemize}
\item[(Step 1)] 
Choose $f_1, \dots, f_n\in\sqzero(\A)$ so that 
$f_1, \dots, f_n$ form a good set of generator of $\grVG^1(\A)$. 
\item[(Step 2)] 
If $\check{\circuit}(f_1, \dots, f_n)$ is the set of signed circuits 
of some oriented matroid, stop. 
Otherwise, return to (Step 1) with different choice of $f_1, \dots, f_n$. 
\end{itemize}
We conjecture that 
the resulting oriented matroid is equivalent to the original one. 
This procedure terminates in finitely many steps whenever 
there are only finitely many choices in (Step 1), e.g., 
the case $\# R^\times<\infty$. 
Unlike the filtered case (Conjecture \ref{conj:algorithm}), 
there can be infinitely many possibilities in (Step 1) when 
$\# R^\times=\infty$, so termination is not guaranteed in full 
generality. 
\end{remark}

\subsection{Concluding remarks}

The graded VG algebra is obviously determined by the 
filtered VG algebra. On the other hand, it remains open 
whether the filtered structure can be reconstructed 
from the graded one. 
\begin{question}
\label{q:reconst}
Let $R$ be an integral domain. Can one recover the 
filtered $R$-algebra $\VG(\A)_R$ from the graded $R$-algebra 
$\grVG^\bullet(\A)_R$ directly? 
\end{question}

By Theorem \ref{thm:main2}, 
the answer is affirmative when $\cha R\neq 2$ and 
$\A$ is generic in codimension $2$ via the reconstruction of 
the oriented matroid. 
It would be interesting to know whether if we can reconstruct 
the filtered VG algebra directly from the graded one. 

In Conjecture \ref{conj:reconst}, we assumed $\cha R\neq 2$. 
Here, we briefly discuss the case $\cha R=2$. 
If $\cha R=2$, we can not recover the oriented matroid 
from the graded VG algebra. The counterexample can be obtained 
again by six generic planes. 
Let $\A_1$ and $\A_2$ be six generic planes in $\R^3$ as in 
Example \ref{ex:6planes}. 
Then $\grVG^\bullet(\A_1)\simeq \grVG^\bullet(\A_2)$ when 
$\cha R=2$, because 
\begin{itemize}
\item 
since both are generic arrangements, we have $L(\A_1)\simeq L(\A_2)$; 
\item it follows 
$\OS^\bullet(\A_1)\simeq\OS^\bullet(\A_2)$ for any coefficient ring $R$; 
\item 
by Proposition \ref{prop:vg1}, when $\cha R=2$, 
$\grVG^\bullet(\A_i)\simeq \OS^\bullet(\A_i)$ ($i=1,2$). 
\end{itemize}
Nevertheless, the oriented matroids are not isomorphic. 

On the other hand, reconstructability from the filtered VG 
algebra with $\cha R=2$ remains unclear. 
\begin{question}
\label{q:reconst2}
Assume $\cha R=2$. Can one recover the oriented matroid 
from $\VG(\A)_R$? 
\end{question}
We expect that a deeper understanding of the relation 
between filtered and graded VG algebras will shed further light 
on oriented matroids and geometry of the complement $M_2(\A)$.

\subsection*{Acknowledgements} 
The authors thank Professor Takuro Abe for pointing 
out several mistakes in the previous version. 
The authors are also grateful to the referee(s) for their careful reading and for constructive comments that improved the manuscript. 
The second author is supported by 
JSPS KAKENHI JP23H00081.

\end{document}